\documentclass[a4paper,12pt,oneside]{amsart}

\usepackage[colorinlistoftodos]{todonotes}
\usepackage[top=3cm,left=2.5cm,right=2.5cm,bottom=3cm]{geometry}
\usepackage{relsize}
\usepackage{hyperref}
\usepackage{amsmath,amsthm,pb-diagram,amssymb,comment}
\usepackage{amsfonts,graphicx,color}
\usepackage{enumitem, fancyhdr, dsfont}
\usepackage{thmtools,cleveref}
\usepackage[normalem]{ulem}
\usepackage{stmaryrd}
\usepackage{textcomp}
\usepackage{contour}
\usepackage{mathbbol}
\usepackage{stmaryrd}
\usepackage[pagewise]{lineno}
\usepackage{tikz,float}
\hypersetup{
    colorlinks=true,
    linkcolor=caribbeangreen,
    filecolor=caribbeangreen,
    urlcolor=caribbeangreen,
    citecolor=caribbeangreen,
}
\definecolor{caribbeangreen}{rgb}{0.0, 0.8, 0.6}

\setlist{topsep=0ex,itemsep=1ex}

    \newcommand{\dom}{\mbox{\rm dom}}

    \newcommand{\thzfc}{\mathrm{ZFC}}

    \newcommand{\Awf}{\mathcal{A}}
    \newcommand{\Bwf}{\mathcal{B}}
    
    \newcommand{\Dwf}{\mathcal{D}}
    \newcommand{\Ewf}{\mathcal{E}}

    \newcommand{\Iwf}{\mathcal{I}}
    \newcommand{\Jwf}{\mathcal{J}}
    \newcommand{\Mwf}{\mathcal{M}}
    \newcommand{\Nwf}{\mathcal{N}}
    
    \newcommand{\Pwf}{\mathcal{P}}
    \newcommand{\Swf}{\mathcal{S}}

 \newcommand{\Wwf}{\mathcal{W}}
    
    \newcommand{\bfrak}{\mathfrak{b}}
    \newcommand{\cfrak}{\mathfrak{c}}
    \newcommand{\dfrak}{\mathfrak{d}}

    \newcommand{\menos}{\smallsetminus}
    
    \newcommand{\pts}{\mathcal{P}}

    \newcommand{\add}{\mbox{\rm add}}
    \newcommand{\cov}{\mbox{\rm cov}}
    \newcommand{\non}{\mbox{\rm non}}
    \newcommand{\cof}{\mbox{\rm cof}}

    \newcommand{\Aor}{\mathbb{A}}
    \newcommand{\Bor}{\mathds{B}}
    \newcommand{\Cor}{\mathds{C}}
    \newcommand{\Dor}{\mathds{D}}
    \newcommand{\Eor}{\mathds{E}}

    \newcommand{\Mor}{\mathds{M}}
    \newcommand{\MIor}{\mathbb{MI}}
    \newcommand{\Por}{\mathds{P}}

    \newcommand{\Ior}{\mathds{I}}

    \newcommand{\Qnm}{\dot{\mathds{Q}}}

    \newcommand{\SNwf}{\mathcal{SN}}

    \newcommand{\cf}{\mbox{\rm cf}}

    \newcommand{\la}{\langle}
    \newcommand{\ra}{\rangle}

\newcommand{\tbf}{\mathbf{t}}

\newcommand{\Seq}{\mathrm{seq}}
\newcommand{\Fr}{\mathrm{Fr}}
\newcommand{\Rbf}{\mathbf{R}}

\newcommand{\Cbf}{\mathbf{C}}
\newcommand{\Lc}{\mathbf{Lc}}

\newcommand{\aLc}{\mathbf{aLc}}

\newcommand{\Scal}{\mathcal{S}}
\newcommand{\id}{\mathrm{id}}
\newcommand{\blc}{\mathfrak{b}^{\mathrm{Lc}}}
\newcommand{\dlc}{\mathfrak{d}^{\mathrm{Lc}}}
\newcommand{\balc}{\mathfrak{b}^{\mathrm{aLc}}}
\newcommand{\dalc}{\mathfrak{d}^{\mathrm{aLc}}}

\newcommand{\abs}[1]{\left|#1\right|}

\newcommand{\leqT}{\preceq_{\mathrm{T}}}
\newcommand{\geqT}{\succeq_{\mathrm{T}}}
\newcommand{\eqT}{\cong_{\mathrm{T}}}

\newcommand{\gen}{\mathrm{gen}}

\newcommand{\refin}[2]{#1 \sqsubseteq_{\mathrm{R}} #2}
\newcommand{\subin}[3]{\mathrm{Sub}(#1, #2, #3)}
\newcommand{\ep}[1]{\mathrm{ep}(#1)}
\newcommand{\st}{\mid}
\newcommand{\set}[2]{\{#1 \st\, #2\}}
\newcommand{\largeset}[2]{\left\{#1 \;\middle|\; #2\right\}}
\newcommand{\seq}[2]{\la #1 \st #2\ra}

\newcommand{\baire}{{}^\omega\omega}
\newcommand{\baireinc}{{}^{\uparrow\omega}\omega}
\newcommand{\cantor}{{}^\omega2}

\newcommand{\NAwf}{\Nwf\!\Awf}

\contourlength{0.8pt}

\contourlength{0.8pt}

\definecolor{sub0}{RGB}{29,32,137}
\definecolor{sub1}{RGB}{1,71,157}
\definecolor{sub2}{RGB}{1,104,183}
\definecolor{sub3}{RGB}{0,160,234}
\definecolor{sug}{RGB}{0,154,68}
\definecolor{suy}{RGB}{208,219,1}

\newcommand{\caribbeangreen}[1]{{\color{caribbeangreen}#1}}

\DeclareSymbolFont{extraup}{U}{zavm}{m}{n}
\DeclareMathSymbol{\varheart}{\mathalpha}{extraup}{86}
\DeclareMathSymbol{\vardiamond}{\mathalpha}{extraup}{87}

\definecolor{dodger}{rgb}{0.0,0.5,1.0}

\definecolor{amber}{rgb}{1.0,0.49,0.0}

\definecolor{ogreen}{RGB}{107,142,35}
\title[Cardinal characteristics associated with small subsets of reals]{Cardinal characteristics associated with small subsets of reals}

\author{Miguel A. Cardona}
\address{Einstein Institute of Mathematics\\
Edmond J. Safra Campus, Givat Ram\\
The Hebrew University of Jerusalem\\
Jerusalem, 91904, Israel}
\email{\href{miguel.cardona@mail.huji.ac.il}{miguel.cardona@mail.huji.ac.il}}
\urladdr{\url{https://sites.google.com/mail.huji.ac.il/miguel-cardona-montoya/home-page
}}

\author{Adam Marton}
 \address{Department of Applied Mathematics and Business Informatics, Faculty of Economics, Technical University of Košice, B. Němcovej 32, 040 01 Košice, Slovakia}
\email{\href{adam.marton@tuke.sk}{adam.marton@tuke.sk}}

\author{Jaroslav \v Supina}
 \address{Institute of Mathematics, P.J. \v{S}af\'arik University in Ko\v sice, Jesenn\'a 5, 040 01 Ko\v{s}ice, Slovakia}
\email{\href{jaroslav.supina@upjs.sk}{jaroslav.supina@upjs.sk}}

\thanks{All authors were supported by the Slovak Research and Development Agency under Contract No.~APVV-20-0045. The first author was supported by Pavol Jozef \v{S}af\'arik University in  Ko\v{s}ice at a postdoctoral position; the second author was supported by the internal faculty grant No. vvgs-2023-2534; and the third author was supported by the grant VEGA 1/0657/22 of the Slovak Grant Agency VEGA}



\begin{document}

\makeatletter
\def\@roman#1{\romannumeral #1}
\makeatother

\newcounter{enuAlph}
\renewcommand{\theenuAlph}{\Alph{enuAlph}}

\numberwithin{equation}{section}
\renewcommand{\theequation}{\thesection.\arabic{equation}}

\theoremstyle{plain}
  \newtheorem{theorem}[equation]{Theorem}
  \newtheorem{corollary}[equation]{Corollary}
  \newtheorem{lemma}[equation]{Lemma}
  \newtheorem{mainlemma}[equation]{Main Lemma}
  \newtheorem{fact}[equation]{Fact}
  \newtheorem{prop}[equation]{Proposition}
  \newtheorem{claim}[equation]{Claim}
  \newtheorem{hopeth}[equation]{Hopeful Theorem}
  \newtheorem{question}[equation]{Question}
  \newtheorem{problem}[equation]{Problem}
  \newtheorem{conjecture}[equation]{Conjecture}
  \newtheorem*{theorem*}{Theorem}
  \newtheorem*{mainthm*}{Main Theorem}
  \newtheorem{teorema}[enuAlph]{Theorem}
  \newtheorem*{corollary*}{Corollary}
  \newtheorem{observation}[equation]{Observation}  
\theoremstyle{definition}
  \newtheorem{definition}[equation]{Definition}
  \newtheorem{example}[equation]{Example}
  \newtheorem{remark}[equation]{Remark}
  \newtheorem{notation}[equation]{Notation}
  \newtheorem{context}[equation]{Context}
  \newtheorem{assumption}[equation]{Assumption}
  \newtheorem*{definition*}{Definition}
  \newtheorem*{acknowledgements*}{Acknowledgements}

\def\sectionautorefname{Section}
\def\subsectionautorefname{Subsection}

\begin{abstract}

Inspired by Bartoszy{\'n}ski's work on small sets, we introduce a new ideal defined by interval partitions on natural numbers and summable sequences of positive reals. Similarly, we present another ideal that relies on Bartoszy{\'n}ski's and Shelah's representation of $F_\sigma$ measure zero sets. We show they are $\sigma$-ideals characterizing all small sets and $F_\sigma$ measure zero sets.
We also study the cardinal characteristics associated with the introduced ideals. We use them to describe the invariants of measure, discuss their connection to Cicho\'n's diagram, and present related consistency results.
\end{abstract}
\maketitle

\makeatother

\section{Introduction}\label{sec:int}

The most important notion in this work is the notion of a \textit{small set} recalled below in \autoref{smallDef}. Introducing small sets is one of the main contributions of Bartoszy{\'n}ski's work~\cite{bartosmall} responding to D. Fremlin's question of whether the cofinality of the covering number of the null ideal is uncountable. This remained an open question for almost two decades. It was in 2000 when Shelah~\cite{Sh00} finally provided a negative answer to the problem. Leading up to that, Bartoszy{\'n}ski~\cite{bartosmall} offered a partial solution to the question by showing that if $\cov(\Nwf) \leq \bfrak,$ then $\cf(\cov(\Nwf)) > \aleph_{0}$. The notion of small sets turned out to be related to the anti-localization cardinals (see e.g. \cite{CM23} for details and original references) and other combinatorial notions such as the \emph{$F_\sigma$ measure zero sets}, see~\cite[Lem.~2.6.3]{bartjud}. 

To present \autoref{smallDef}, given a formula $\phi$, we use ``$\exists^\infty\, n\in\omega\colon \phi$" to abbreviate ``infinitely many natural numbers satisfy $\phi$". 

\begin{definition}[{\cite{bartosmall,BART2018}}]\label{smallDef}
Let $X\subseteq\cantor$.
\begin{enumerate}
    \item \textit{$X$ is small} if, there are sequences $\seq{I_n, J_n}{n \in\omega}$ such that 
\begin{enumerate}[label=(\alph*)]
    \item $\set{I_n}{n\in\omega}\subseteq [\omega]^{<\aleph_0}$ is a partition of a cofinite subset of $\omega$,

    \item $J_n \subseteq {}^{I_n}2$ for all $n$,
    \item $\sum_{n<\omega}\frac{|J_n|}{2^{|I_n|}}<\infty$, and 
    \item $X\subseteq\set{x\in\cantor}{\exists^{\infty}n\in\omega:x{\upharpoonright}I_n\in J_n}$. 

\end{enumerate}
    \item \textit{$X$ is small$^\star$} if, in addition, sets $I_n$ are disjoint intervals, that is, if there exists a strictly increasing sequence of integers $\set{k_n}{n\in\omega}$ such that $I_n=[k_n, k_{n+1})$ for $n\in\omega$.
\end{enumerate}
\item 
Denote by $\Swf$ and $\Swf^\star$ the collection of all small sets and  small$^\star$ sets in $\cantor$, respectively. It is clear that $\Swf^\star\subseteq\Swf$ and that every small set has measure zero, i.e.,
$\Swf\subseteq \Nwf$ (see \autoref{notation}). Bartoszy{\'n}ski~\cite[Lem.~1.3]{bartosmall} proved that none of them is an ideal.  
\end{definition}

Let $\Ior$ denote the family of all partitions of $\omega$ into finite non-empty intervals\footnote{For technical reasons, it is sometimes convenient to consider a partition of some cofinite subset of $\omega$ instead of a partition of $\omega$. Thus, depending on the context, $I\in\Ior$ will sometimes mean that $I$ is a partition of some cofinite subset of $\omega$.} and let
\[\ell^1_+:=\largeset{\varepsilon\in{^\omega(0,\infty)}}{\sum_{n\in\omega}\varepsilon_n<\infty}.\]

In this work, we offer a characterization of $\Swf^\star$ in terms of  a new 
$\sigma$-ideal $\Swf_{I,\varepsilon}$ (treating $\{\emptyset\}$ as a $\sigma$-ideal by convention) parametrized by $I\in\Ior$ and $\varepsilon\in\ell^1_+$, i.e,
\[   \Swf^\star=\bigcup_{(I,\varepsilon)\in\Ior\times\ell^1_+}\Swf_{I, \varepsilon}.
    \]
Our ideal $\Swf_{I,\varepsilon}$ is motivated by the definition of small sets, recalled in our \autoref{smallDef}.
The introduction of $\Swf_{I,\varepsilon}$ itself and its properties are provided in~\autoref{sec:a0}.

Recall that given an ideal $\Iwf$ of subsets of $X$ such that $\{x\}\in \Iwf$ for all $x\in X$, \emph{the cardinal characteristics associated with $\Iwf$} are defined by
\begin{align*}
 \add(\Iwf)&:=\min\set{|\Jwf|}{\Jwf\subseteq\Iwf\text{\ and } \bigcup\Jwf\notin\Iwf};\\
     \cov(\Iwf)&:=\min\set{|\Jwf|}{\Jwf\subseteq\Iwf\text{\ and }\bigcup\Jwf=X};\\
    \non(\Iwf)&:=\min\set{|A|}{A\subseteq X\text{\ and }A\notin\Iwf};\\
     \cof(\Iwf)&:=\min\set{|\Jwf|}{\Jwf\subseteq\Iwf\text{\ is cofinal in }\la\Iwf,\subseteq\ra}.   
\end{align*}

These cardinals are referred to as additivity, covering, uniformity, and cofinality of $\Iwf$, respectively.

In order to present our main results, we review the basic notation:

\begin{notation}\label{notation}\

\begin{enumerate}[label=\rm(\arabic*)]
    \item Given a formula $\phi$, ``$\forall^\infty\, n\in\omega\colon \phi$" means that all but finitely many natural numbers satisfy $\phi$; ``$\exists^\infty\, n\in\omega\colon \phi$" means that infinitely many natural numbers satisfy $\phi$.

    \item Denote by $\Nwf$ and $\Mwf$ the $\sigma$-ideals of Lebesgue null sets and of meager sets in~$\cantor$, respectively, and let $\Ewf$ be the $\sigma$-ideal generated by the closed measure zero subsets of $\cantor$. It is well-known that $\Ewf\subseteq\Nwf\cap\Mwf$. Moreover, it was proved that this inclusion is proper (see~\cite[Lemma 2.6.1]{bartjud}). 
    \item $\cfrak:=2^{\aleph_0}$.

    \item For $f,g\in\baire$ define 
$$f\leq^*g\text{ iff } \forall^\infty n\in\omega: f(n)\leq g(n).$$

    \item We let
\begin{align*}
    \bfrak &:=\min\set{\abs{F}}{F\subseteq\baire\text{ and }\forall g\in\baire\ \exists f\in F:f\not\leq^* g},\\
    \dfrak &:=\min\set{\abs{D}}{D\subseteq\baire\text{ and }\forall g\in\baire\ \exists f\in D:g\leq^* f},
\end{align*}
the \textit{bounding number} and the \textit{dominating number}, respectively.

\end{enumerate}
\end{notation}

One of our first main results exhibits that the covering and the uniformity number of our ideal are related to anti-localization cardinals $\balc_{b,h}, \dalc_{b,h}$ (see \autoref{defloc}):  

\begin{teorema}[\autoref{thm:alocsmall}]\label{Thm:a1}
$\nu_i=\kappa_i$ for $i<2$, where 
\begin{itemize}
    \item $\nu_0=\min\set{\cov(\Swf_{I, \varepsilon})}{I\in\Ior\text{\ and\ }\varepsilon\in\ell^1_+}$,
      \item $\nu_1=\sup\set{\non(\Swf_{I, \varepsilon})}{I\in\Ior\text{\ and\ }\varepsilon\in\ell^1_+}$,
       \item $\kappa_0=\min\largeset{\balc_{b,h}}{b,h\in{^\omega\omega}\text{\ and\ } \sum_{n\in\omega}\frac{h(n)}{b(n)}<\infty}$,
        \item  $\kappa_1=\sup\largeset{\dalc_{b,h}}{b,h\in{^\omega\omega}\text{\ and\ } \sum_{n\in\omega}\frac{h(n)}{b(n)}<\infty}$.
\end{itemize}
\end{teorema}

As a consequence of \autoref{Thm:a1}, we obtain an alternative characterization of $\cov(\Nwf)$ and $\non(\Nwf)$. We also establish an analogous characterization of $\add(\Nwf)$ and $\cof(\Nwf)$.

\begin{teorema}[\autoref{cov_non}, \autoref{Thm:cofadd}]\label{ThmM:ac}\
\begin{enumerate}[label=\rm(\alph*)]

    \item  Assume that $\cov(\Nwf)<\bfrak$. Then $\cov(\Nwf)=\min\set{\cov(\Swf_{I, \varepsilon})}{I\in\Ior\text{\ and\ }\varepsilon\in\ell^1_+}$.

    \item Assume that $\non(\Nwf)>\dfrak$. Then $\non(\Nwf)=\sup\set{\non(\Swf_{I, \varepsilon})}{I\in\Ior\text{\ and\ }\varepsilon\in\ell^1_+}$.
    
   \item\label{ThmM:ac:a} Assume that $\add(\Nwf)<\bfrak$. Then $\add(\Nwf)=\min\set{\add(\Swf_{I,\varepsilon})}{I\in\Ior\text{\ and\ }\varepsilon\in\ell^1_+}$.

    \item\label{ThmM:ac:b}  Assume that $\cof(\Nwf)>\dfrak$. Then $\cof(\Nwf)=\sup\set{\cof(\Swf_{I, \varepsilon})}{I\in\Ior\text{\ and\ }\varepsilon\in\ell^1_+}$.
\end{enumerate}    
\end{teorema}

Inspired by the combinatorial characterization of $F_\sigma$ measure zero sets due to Bartoszy{\'n}ski and Shelah~\cite{BS1992} stated below as \autoref{a13}, we also introduce a $\sigma$-ideal $\Ewf_{I,\varepsilon}$ para\-metrized in the same way as $\Swf_{I,\varepsilon}$, which can be used to redescribe $\Ewf$ as
\[\Ewf=\bigcup_{(I,\varepsilon)\in\Ior\times\ell^1_+}\Ewf_{I, \varepsilon}.\] Details are developed in~\autoref{sec:a0}.

\begin{theorem}[{{\cite[Thm~4.3]{BS1992}}}]\label{a13}
For every $X\subseteq\cantor$,  \textit{$X\in\Ewf$} iff there is a partition $\seq{I_n}{n\in\omega}$ of $\omega$ into finite, disjoint intervals and a sequence $\seq{J_n}{n \in\omega}$ such that 
\begin{enumerate}[label=\normalfont(\alph*)]
    \item $J_n \subseteq {}^{I_n}2$ for all $n$,
    \item $\sum_{n<\omega}\frac{|J_n|}{2^{|I_n|}}<\infty$, and 
    
    \item $X\subseteq\set{x\in\cantor}{\forall^{\infty}n:x{\upharpoonright}I_n\in J_n}$.
\end{enumerate}
\end{theorem}
Note that it follows directly from \autoref{a13} that $\Ewf\subseteq\Swf^\star$.

We have found a connection between the localization cardinals  $\blc_{b,h}, \dlc_{b,h}$ (see \autoref{defloc}) and the covering and uniformity numbers of $\Ewf_{I,\varepsilon}$, similar to the connection established in \autoref{Thm:a1}. 

\begin{teorema}[\autoref{a11}]\label{Thm:m2} $\lambda_i=\theta_i$ for $i<2$ where 
\begin{itemize}
    \item $\theta_0=\min\set{\cov(\Ewf_{I, \varepsilon})}{I\in\Ior\text{\ and\ }\varepsilon\in\ell^1_+}$,
    \item $\theta_1= \sup\set{\non(\Ewf_{I, \varepsilon})}{I\in\Ior\text{\ and\ }\varepsilon\in\ell^1_+}$,
    \item $\lambda_0=\min\largeset{\dlc_{b,h}}{b,h\in{^\omega\omega}\text{\ and\ } \sum_{n\in\omega}\frac{h(n)}{b(n)}<\infty}$,
    \item $\lambda_1=\sup\largeset{\blc_{b,h}}{b,h\in{^\omega\omega}\text{\ and\ } \sum_{n\in\omega}\frac{h(n)}{b(n)}<\infty}$.
\end{itemize}
\end{teorema}

Concerning the cardinals $\theta_i$ in the previous theorem, it turns out these are related to the covering and the uniformity of the ideal of null-additive sets (see~\autoref{def:na}):

\begin{teorema}[\autoref{NAbound}] Let $\theta_i$ be as in~\autoref{Thm:m2}. Then
\begin{enumerate}[label=\normalfont(\alph*)]
    \item  $\non(\NAwf)\leq \theta_1$. 
    \item $\cov(\NAwf)\geq \theta_0$.
\end{enumerate}
\end{teorema}

In~\autoref{sec:zfc}, we clarify the relation between the ideal $\Ewf_{I,\varepsilon}$ and $\Swf_{I,\varepsilon}$ as follows. 
\begin{teorema}[\autoref{SEcomparison}]\label{a12}
 $\cof(\Swf_{I,\varepsilon})=\cof(\Ewf_{I,\varepsilon})$ and      $\add(\Ewf_{I,\varepsilon})=\add(\Swf_{I,\varepsilon})$ for reasonable pairs $(I,\varepsilon)\in\Ior\times\ell^1_+$.   
\end{teorema}

With respect to Cicho\'n’s diagram (see~\autoref{cichonext}), in $\thzfc$, we prove:

\begin{teorema}\label{Thm:f} \ 
\begin{enumerate}[label=\rm(\alph*)]
    \item\label{Thm:f:a} $\add(\Nwf)\leq\add(\Ewf_{I,\varepsilon})=\add(\Swf_{I,\varepsilon})\leq\cof(\Swf_{I,\varepsilon})=\cof(\Ewf_{I,\varepsilon})\leq\cof(\Nwf)$.
    \item\label{Thm:f:b}  $\cov(\Nwf)\leq\cov(\Swf)\leq\cov(\Swf_{I,\varepsilon})\leq\cov(\Ewf_{I,\varepsilon})$ and $\non(\Ewf_{I,\varepsilon})\leq\non(\Swf_{I,\varepsilon})\leq\non(\Swf)\leq\non(\Nwf)$.

    \item\label{Thm:f:c}  $\cov(\Ewf)\leq\cov(\Ewf_{I,\varepsilon})$ and $\non(\Ewf_{I,\varepsilon})\leq\non(\Ewf)$.
\end{enumerate}
\end{teorema}

Item~\ref{Thm:f:a} in the theorem is obtained by combining~\autoref{SEcomparison} and \autoref{a10}, while items~\ref{Thm:f:b} and \ref{Thm:f:c} follow from the inclusion relation between the ideals. 

Regarding inequalities related to the cardinal characteristics associated with our ideals, \autoref{Thm:f} seems to be the most optimal: In~\autoref{sec:cs}, we manage to show that, in most cases, no further inequalities can be proved with the cardinals in Cicho\'n’s diagram.

\begin{figure}[ht]
\centering
\begin{tikzpicture}[scale=1.15]
\small{
\node (aleph1) at (-1,1.3) {$\aleph_1$};
\node (addn) at (-0.25,2.3){$\add(\Nwf)$};
\node (covn) at (-0.25,7.5){$\cov(\Nwf)$};
\node (covsie) at (0.75, 6){\caribbeangreen{$\cov(\Swf_{I,\varepsilon})$}};
\node (covs) at (2.3,8.5){$\cov(\Swf)$};
\node (nonn) at (11,2.3) {$\non(\Nwf)$};
\node (cfn) at (11,7.5) {$\cof(\Nwf)$};
\node (addm) at (4.2,2.3) {$\add(\Mwf)=\add(\Ewf)$};
\node (covm) at (6.6,2.3) {$\cov(\Mwf)$};
\node (nons) at (8.4,1.3) {$\non(\Swf)$};
\node (nonm) at (4.2,7.5) {$\non(\Mwf)$};
\node (cfm) at (6.6,7.5) {$\cof(\Mwf)=\cof(\Ewf)$};
\node (b) at (4.2,5.1) {$\bfrak$};
\node (d) at (6.6,5.1) {$\dfrak$};
\node (c) at (11.7,8.5) {$\cfrak$};
\node (none) at (7.45,5.1) {$\non(\Ewf)$};
\node (cove) at (3.3,5.1) {$\cov(\Ewf)$};

\node (addes) at (2,4.2) {\caribbeangreen{$\add(\Ewf_{I,\varepsilon})=\add(\Swf_{I,\varepsilon})$}};
\node (cofes) at (8.8,6) {\caribbeangreen{$\cof(\Ewf_{I,\varepsilon})=\cof(\Swf_{I,\varepsilon})$}};
\node (coveie) at (2,6.8) {\caribbeangreen{$\cov(\Ewf_{I,\varepsilon})$}};
\node (nonsie) at (10.2, 4){\caribbeangreen{$\non(\Swf_{I,\varepsilon})$}};
\node (noneie) at (8.8,3.2) {\caribbeangreen{$\non(\Ewf_{I,\varepsilon})$}};
\draw (addn) edge[->] (addes);
\draw (cofes) edge[->] (cfn);

   \draw (noneie) edge[->] (cofes);
   \draw (nonsie) edge[->] (cofes);
   
  \draw (covn) edge[->] (covs);
 \draw (covn) edge[line width=.15cm,white,-] (covsie);
 \draw (covn) edge[->] (covsie);
  
   \draw (noneie) edge[->] (nonsie);
 
 \draw (addes) edge[->] (covsie);
 \draw (aleph1) edge[->] (addn);
  \draw    (addn) edge[->] (covn)
      (covn) edge [->] (nonm)
      (covm) edge [->] (nonn)
      (nonm)edge [->] (cfm)
      (cfm)edge [->] (cfn)
      (cfn) edge[->] (c);

\draw  (nons) edge [->]  (nonn)
(covs) edge [->]  (c)
(aleph1) edge [->]  (nons);

\draw  (cove) edge [<-, ]  (covn)
(cove) edge [->]  (coveie);

\draw (noneie) edge [->]  (none);

\draw
   (cove) edge [->]  (cfm) 
   (cove) edge [<-]  (addm)
   (none) edge [<-]  (addm)
    (none) edge [->]  (cfm);
\draw
   (addn) edge [->]  (addm)
   (addm) edge [->]  (covm)
   (nonn) edge [->]  (cfn);
\draw (addm) edge [->] (b)
      (b)  edge [->] (nonm);
\draw (covm) edge [->] (d)
      (d)  edge[->] (cfm);
\draw (b) edge [->] (d);

\draw (none) edge[line width=.15cm,white,-]  (nonn);
\draw (none) edge [->]  (nonn);

\draw  (nonsie) edge[line width=.15cm,white,-]  (nons);
\draw  (nonsie) edge [->]  (nons);

\draw (none) edge[line width=.15cm,white,-]  (nons);
\draw (none) edge [->]  (nons);

\draw(addm) edge[line width=.15cm,white,-]  (none) ;
\draw(addm) edge [->]  (none) ;

\draw (none) edge[line width=.15cm,white,-]  (nonm);
\draw (none) edge [->]  (nonm);

\draw(cove) edge[line width=.15cm,white,-]  (cfm);
\draw(cove) edge [->]  (cfm) ;

\draw (addes) edge[line width=.15cm,white,-] (noneie);
\draw (addes) edge[->] (noneie);

\draw (cove) edge[line width=.15cm,white,-]  (covm);
\draw (cove) edge [<-]  (covm);

\draw (coveie) edge[line width=.15cm,white,-] (cofes);
\draw (coveie) edge[->] (cofes);

\draw  (covs) edge[line width=.15cm,white,-]  (cove);
\draw  (covs) edge [->]  (cove);

\draw (addes) edge[line width=.15cm,white,-] (coveie);
\draw (addes) edge[->] (coveie);

\draw  (covs) edge[line width=.15cm,white,-]  (covsie);
\draw  (covs) edge [->]  (covsie);

\draw (covsie) edge[line width=.15cm,white,-] (coveie);
\draw (covsie) edge[->] (coveie); 
}
\end{tikzpicture}
\caption{Cicho\'n's diagram including the cardinal characteristics associated with our ideals, and $\add(\Ewf) = \add(\Mwf)$ and $\cof(\Ewf) = \cof(\Mwf)$ due to Bartoszy{\'n}ski and Shelah~\cite{BS1992}. The relations hold for contributive pairs.}
\label{cichonext}
\end{figure}

\section{Layering of small sets}\label{sec:a0}
The objective of this section is to introduce and present the basic properties of the two new ideals $\Swf_{I,\varepsilon}$ and $\Ewf_{I,\varepsilon}$. Their definitions, as discussed in~\autoref{sec:int}, are primarily based on the combinatorial description of small sets and $F_\sigma$ measure zero sets. We examine $\Swf_{I,\varepsilon}$ and $\Ewf_{I,\varepsilon}$ and in \autoref{a2} we demonstrate that they are, in fact, $\sigma$-ideals. We also show that $\Swf^\star$ and $\Ewf$ can be characterized using these $\sigma$-ideals.

The families $\Swf_{I,\varepsilon}$ and $\Ewf_{I,\varepsilon}$ are defined as follows. 

\begin{definition}\label{a1}
\mbox{}
\begin{enumerate}[label=\normalfont(\alph*)]
    \item\label{a1a} Let $I\in\Ior$ and $\varepsilon\in\ell^1_+.$ Let us define 
\[\Sigma_{I,\varepsilon}=\largeset{\varphi\in\prod_{n\in\omega}\Pwf({}^{I_n}2)}{\lim_{n\to\infty}\frac{|\varphi(n)|}{2^{|I_n|}\cdot\varepsilon_n}=0}.\]

\item\label{a1b} For $\varphi\in \Sigma_{I,\varepsilon}$ define 
\[[\varphi]_\infty=\set{x\in\cantor}{\exists^\infty n\in\omega:x{\upharpoonright}I_n\in \varphi(n)},\]
\[[\varphi]_*=\set{x\in\cantor}{\forall^\infty n\in\omega:x{\upharpoonright}I_n\in \varphi(n)}.\]
\item Given $I$ and $\varepsilon$ as in~\ref{a1a} we define 
\[\Swf_{I,\varepsilon}=\set{X\subseteq\cantor}{\exists\varphi\in\Sigma_{ I, \varepsilon}:X\subseteq[\varphi]_\infty},\]
\[\Ewf_{I,\varepsilon}=\set{X\subseteq\cantor}{\exists\varphi\in\Sigma_{ I, \varepsilon}:X\subseteq[\varphi]_*}.\]
\end{enumerate}
\end{definition}

Concerning \autoref{a1}~\ref{a1a}, we shall often use in the proofs that 
\[
\lim_{n\to\infty}\frac{|\varphi(n)|}{2^{|I_n|}\cdot\varepsilon_n}=0\ \text{ iff }\  \forall N<\omega\ \forall^\infty n\in\omega:\frac{|\varphi(n)|}{2^{|I_n|}}<\frac{\varepsilon_n}{N}.
\]

We first show that the families $\Swf_{I,\varepsilon}$ and $\Ewf_{I,\varepsilon}$ are $\sigma$-ideals (treating the trivial cases as $\sigma$-ideals). Moreover, these $\sigma$-ideals completely characterize families $\Swf^\star$ and $\Ewf$. Note that $\Swf_{I,\varepsilon}\subseteq\Swf^\star$ and $\Ewf_{I,\varepsilon}\subseteq\Ewf$ for any $I\in\Ior$ and $\varepsilon\in\ell^1_+.$ 

Next, we shall need a~technical lemma.
\begin{lemma}\label{a9}
      For any $\varepsilon\in\ell^1_+$ there is a non-decreasing $\delta\in {^\omega(0,\infty)}$ such that $\delta\to\infty$ and $\sum_{j\in\omega}\delta_j\varepsilon_j<\infty$.
      \end{lemma}
  \begin{proof}
     Let $s=\sum_{j\in\omega}\varepsilon_j$. Next, define an increasing sequence $\seq{n_i}{i\in\omega}$ as follows: let $n_0=0$ and for $i\geq 1$ let 
     $$n_{i}=\min\largeset{k>n_{i-1}}{\sum_{n\geq k}2^i\varepsilon_n<\frac{s}{2^i}}.$$
     For $j\in [n_i,n_{i+1})$ define $\delta_j=2^{i}$. To conclude the proof it is enough to show \[\lim_{i\to\infty}\sum_{j<n_i}\delta_j\varepsilon_j<\infty.\] For any $i\in\omega$ we have
     $$\sum_{j<n_i}\delta_j\varepsilon_j= \sum_{k < i}\sum_{j=n_k}^{n_{k+1}}2^k\varepsilon_j<\sum_{k<i}\frac{s}{2^k}<2s.$$
     The sequence $a_i=\sum_{j<n_i}\delta_j\varepsilon_j$ is monotone and bounded from above, hence convergent.
  \end{proof}

From now on, whenever $a<b\leq\omega$, we let $[a,b)$ denote the integer interval $\set{n\in\omega}{a\leq n<b}$.

\begin{theorem}\label{a2}
\mbox{}
\begin{enumerate}[label=\normalfont(\alph*)]
    \item\label{a2a} $\Swf_{I,\varepsilon}$ is a $\sigma$-ideal for any $I\in \Ior$ and $\varepsilon\in \ell^1_+$.

 \item\label{a2aa} $\Ewf_{I,\varepsilon}$ is a $\sigma$-ideal for any $I\in \Ior$ and $\varepsilon\in \ell^1_+$.
    \item\label{a2b} $X\in\Swf^\star$ iff there is an $I\in\Ior$ and an $\varepsilon\in\ell^1_+$ such that $X\in\Swf_{I, \varepsilon}$. In particular, 
    \[
    \Swf^\star=\bigcup_{(I,\varepsilon)\in\Ior\times\ell^1_+}\Swf_{I, \varepsilon}.
    \]
    \item\label{a2bb} $X\in\Ewf$ iff there is an $I\in\Ior$ and an $\varepsilon\in\ell^1_+$ such that $X\in\Ewf_{I, \varepsilon}$. In particular, 
    \[
    \Ewf=\bigcup_{(I,\varepsilon)\in\Ior\times\ell^1_+}\Ewf_{I, \varepsilon}.
    \]
\end{enumerate}
\end{theorem}
\begin{proof}
To prove \ref{a2a} and \ref{a2aa} it suffices to prove the following:
\begin{claim}
Given $\set{\varphi^n}{n\in\omega}\subseteq \Sigma_{I,\varepsilon}$ there is  $\varphi\in\Sigma_{I,\varepsilon}$ such that $$\forall n\in\omega\ \forall^\infty j\in\omega: \varphi^n(j)\subseteq \varphi(j).$$
\end{claim}
\begin{proof}
  Define a sequence $\seq{k_n}{n\in\omega}$ as follows: $k_0=0$ and for $n\geq 1$,
$$k_{n}=\min\largeset{m>k_{n-1}}{\forall i< n\forall j\geq m:\frac{\abs{\varphi^i(j)}}{2^{\abs{I_j}}}<\frac{\varepsilon_j}{n^2}}.$$

Now define $\varphi$ by 
\[\varphi(j)=\bigcup_{i < n}\varphi^i(j)\]
for all $j\in[k_n,k_{n+1})$. Notice that
\begin{itemize}
    \item for all $j\in[k_n,k_{n+1})$,
\[\frac{|\varphi(j)|}{2^{|I_j|}}\leq\sum_{i=0}^{n-1}\frac{|\varphi^i(j)|}{2^{|I_j|}}< n\cdot\frac{\varepsilon_j}{n^2}=\frac{\varepsilon_j}{n},\]

    \item $\forall n\in\omega\,\forall^\infty j\in\omega:\varphi^n(j)\subseteq \varphi(j )$.
\end{itemize}  
\end{proof}

\noindent\ref{a2b}: 
The direction from right to left is straightforward. We prove the reversed direction. To see that, let $Y\in \Swf^\star$. Choose $I\in\Ior$ and $\varphi\in \prod_{n\in\omega}\Pwf({^{I_n}2})$ such that $\sum_{n\in\omega}\frac{\abs{\varphi(n)}}{2^{\abs{I_n}}}<\infty$ and $Y\subseteq \set{x\in {^\omega2}}{\exists^\infty n:x{\restriction} I_n\in\varphi(n)}$. Applying~\autoref{a9} we find $\delta\in {^\omega(0,\infty)}$ such that $\delta\to\infty$ and $$\sum_{n\in\omega}\delta_n\frac{\abs{\varphi(n)}}{2^{\abs{I_n}}}<\infty.$$
  Define for any $n\in\omega$ $$\varepsilon_n=\delta_n\frac{\abs{\varphi(n)}}{2^{\abs{I_n}}}.$$ Then $\varepsilon\in\ell^1_+$ and
  $$\frac{\abs{\varphi(n)}}{2^{\abs{I_n}}\cdot \varepsilon_n}=\frac{1}{\delta_n}\to 0.$$
  Thus, $\varphi\in \Sigma_{I,\varepsilon}$ and $Y\in \Swf_{I,\varepsilon}$.

  \noindent\ref{a2bb}: Like in~\ref{a2b}, using \autoref{a13}.
\end{proof}

For many $I,\varepsilon$ we have that $\Swf_{I,\varepsilon}=\{\emptyset\}$ or $\Ewf_{I,\varepsilon}=\{\emptyset\}$. 
Since we are interested in non-trivial $\sigma$-ideals, we provide the following characterizations, \autoref{charofrSIe} and \autoref{charofrEIe}, to avoid the trivial cases. 

\begin{lemma}\label{charofrSIe}
    The following statements are equivalent for any $(I,\varepsilon)\in\Ior\times\ell^1_+$.
    \begin{enumerate}[label=\normalfont(\arabic*)]
        \item\label{charofrSIe1} $\Swf_{I,\varepsilon}\neq \{\emptyset\}$.
         \item\label{charofrSIe1lim} $\exists A\in [\omega]^\omega: \lim_{n\in A} 2^{\abs{I_{n}}}\cdot \varepsilon_{n}=\infty$.

    \end{enumerate}
\end{lemma}
\begin{proof}
We shall use the following representation of \ref{charofrSIe1lim}:
\[
\exists A\in [\omega]^\omega\ \forall N\in\omega\ \forall^\infty n\in A:N\cdot 2^{-\abs{I_n}}<\varepsilon_n.
\] 

\ref{charofrSIe1}$~\Rightarrow$~\ref{charofrSIe1lim}: Suppose not~\ref{charofrSIe1lim}. Then for any $A\in[\omega]^\omega$ there is $N\in\omega$ such that 
\[\exists^\infty n\in A: \frac{1}{2^{\abs{I_n}}} \geq \frac{\varepsilon_n}{N}.\]
Let $\varphi\in \Sigma_{I,\varepsilon}$ be such that $\emptyset\neq [\varphi]_\infty\in \Swf_{I,\varepsilon}$. There exist infinitely many $n$ such that $\varphi(n)\neq \emptyset$. Let $A=\set{n\in\omega}{\varphi(n)\neq \emptyset}$. Since $\abs{\varphi(n)}\geq 1$  for any $n\in A$, we have that $\frac{\abs{\varphi(n)}}{2^{\abs{I_n}}}\geq \frac{1}{2^{\abs{I_n}}}$ for all $n\in A$. Therefore, 
$$\exists^\infty n\in \omega: \frac{\abs{\varphi(n)}}{2^{\abs{I_n}}} \geq \frac{\varepsilon_n}{N}.$$ Consequently, $\varphi\notin \Sigma_{I,\varepsilon}$, which is a contradiction.

\ref{charofrSIe1lim}$~\Rightarrow$~\ref{charofrSIe1}: Define $\varphi\in \Sigma_{I,\varepsilon}$ such that $\abs{\varphi(n)}=1$ for all $n\in A$ and $0$ otherwise. We will show that $\varphi\in \Sigma_{I,\varepsilon}$. Let $N\in\omega$. Then by \ref{charofrSIe1lim},

\begin{align*}
   &\forall^\infty n\in A:\frac{\abs{\varphi(n)}}{2^{\abs{I_n}}}=\frac{1}{2^{\abs{I_n}}}<\frac{\varepsilon_n}{N},\text{ and} \\
  & \forall^\infty n\in \omega\smallsetminus A: \frac{\abs{\varphi(n)}}{2^{\abs{I_n}}}=0<\frac{\varepsilon_n}{N}.
\end{align*}
Therefore, $\varphi\in\Sigma_{I,\varepsilon}$ and clearly, $[\varphi]_\infty\neq \emptyset$.
\end{proof}

Any pair $(I,\varepsilon)\in\Ior\times\ell^1_+$ satisfying the conditions in \autoref{charofrSIe} is called $\Swf^\star$\textbf{-contributive}.

Because of \autoref{charofrSIe}, we have the following.

\begin{corollary}\label{reasonablePairs}
   For any $(I,\varepsilon)\in\Ior\times\ell^1_+$,  if the sequence $\seq{2^{\abs{I_n}}\cdot\varepsilon_n}{n\in\omega}$ is bounded, then $\Swf_{I,\varepsilon}=\{\emptyset\}$.
\end{corollary}

\begin{remark}
Let $(I,\varepsilon)\in\Ior\times\ell^1_+$.

\begin{enumerate}[label=\normalfont(\alph*)]
\item    If the sequence $\seq{\abs{I_n}}{n\in\omega}$ is bounded, then $\Swf_{I,\varepsilon}=\Ewf_{I,\varepsilon}=\{\emptyset\}$. This is a consequence of \autoref{reasonablePairs}.
 \item\label{boundedsub}  Let $\varphi\in\Sigma_{I,\varepsilon}$. If there is $A\in[\omega]^\omega$ such that  the sequence $\seq{\abs{I_n}}{n\in A}$ is bounded, then $\forall^\infty n\in A : \varphi(n)=\emptyset$.
 Indeed, if there existed infinitely many $n\in A$ such that $\varphi(n)\neq \emptyset$, then  for infinitely many $n\in A$
$$\frac{\abs{\varphi(n)}}{2^{\abs{I_n}}\cdot \varepsilon_n}\geq \frac{\abs{\varphi(n)}}{2^B\cdot \varepsilon_n}\geq \frac{1}{2^B}\cdot \frac{1}{\varepsilon_n}\to\infty,$$
where $B$ is a bound of $\seq{\abs{I_n}}{n\in A}$.
Since a subsequence of $\frac{\abs{\varphi(n)}}{2^{\abs{I_n}}\cdot \varepsilon_n}$ tends to infinity, $\lim_{n\to\infty }\frac{\abs{\varphi(n)}}{2^{\abs{I_n}}\cdot \varepsilon_n}=0$ cannot hold.
 \end{enumerate}
\end{remark}

Just as in \autoref{charofrSIe}, we have the following for $\Ewf_{I,\varepsilon}$.

\begin{lemma}\label{charofrEIe}
    The following statements are equivalent for any $(I,\varepsilon)\in\Ior\times\ell^1_+$.
    \begin{enumerate}[label=\normalfont(\arabic*)]
        \item\label{charofrEIe1} $\Ewf_{I,\varepsilon}\neq \{\emptyset\}$.
        \item\label{charofrEIe1lim} $\lim_{n\to\infty} 2^{\abs{I_n}}\cdot \varepsilon_n=\infty$.
    \end{enumerate}
\end{lemma}
\begin{proof}
Similar to the proof of \autoref{charofrSIe}.
\end{proof}
A pair $(I,\varepsilon)\in\Ior\times\ell_+^1$ satisfying conditions in \autoref{charofrEIe} is called $\Ewf$\textbf{-contributive}. 
Notice that any $\Ewf$-contributive $(I,\varepsilon)$ is automatically also $\Swf^\star$-contributive.

Inclusions between the studied ideals are depicted in \autoref{diagram:rel}.
\begin{figure}[H]
\begin{center}
\begin{tikzpicture}[]
\node (a) at (-5, -3.5) {$\Ewf_{I,\varepsilon}$};
\node (as) at (-5, -1) {$\Swf_{I,\varepsilon}$};
\node (aa) at (-3, -3.5) {$\Ewf$};
\node (aas) at (-3, -1) {$\Swf^\star$};
\node (b) at (-1, -3.5) {$\Mwf\cap\Nwf$};
\node (bs) at (-1, -1) {$\Nwf$};
\node (c) at (1, -3.5) {$\Mwf$};
\foreach \from/\to in {a/as, aa/aas,a/aa,as/aas,aa/b,aas/bs,b/bs,b/c} \draw [->] (\from) -- (\to);
\end{tikzpicture}
\end{center}
\caption{Inclusions.}
\label{diagram:rel}
\end{figure}

Many inclusions in \autoref{diagram:rel} are proper or cannot be added as one can see in the following assertion. Moreover, it was shown in \cite{bartosmall} that the inclusion $\Swf^\star\subseteq\Nwf$ is proper, i.e., $\Swf^\star\subsetneq\Nwf$.
 \begin{prop}\label{relation:SandE}
 The following holds for any $I\in\Ior$ and any $\varepsilon\in\ell^1_+$:
\begin{enumerate}[label=\normalfont(\alph*)]
    \item\label{neqIea} $\Swf_{I,\varepsilon}\subsetneq \Swf^\star$.
    \item\label{neqIeb}  $\Ewf_{I,\varepsilon}\subsetneq \Ewf$.\\
    \par
\noindent Moreover, if $(I,\varepsilon)$ is $\Swf^\star$-contributive then    
    \item\label{neqIec}  $\Ewf_{I,\varepsilon}\subsetneq \Swf_{I,\varepsilon}$.    
    \item\label{neqIed}  $\Swf_{I,\varepsilon}\not\subseteq\Ewf$.        
\end{enumerate} 
\end{prop}
 \begin{proof}
 \ref{neqIea}: This follows from the fact that $\Swf_{I,\varepsilon}$ is a $\sigma$-ideal by \autoref{a2} (recall that $\Swf^\star$ is not an ideal see~\cite[Lem.~1.3]{bartosmall} or see also~\cite{BART2018,bartjud}). 

\ref{neqIeb}:
Because of \autoref{charofrEIe}, we can restrict ourselves to pairs $(I,\varepsilon)\in\Ior\times\ell^1_+$ such that $\forall^\infty n\in\omega: 2^{-\abs{I_n}}<\varepsilon_n$. So assume that $(I,\varepsilon)$ is such a pair. We will construct a $\bar\varphi$  such that $[\bar\varphi]_*\in\Ewf\smallsetminus\Ewf_{I,\varepsilon}$. 

Let \[v:=\set{n\in\omega}{2^{-\abs{I_n}}\leq \varepsilon_n}\ \textrm{and}\ w:=\set{n\in\omega}{\frac{\abs{\varphi(n)}}{2^{\abs{I_n}}}<\varepsilon_n<1}\footnote{$w=\set{n\in\omega}{\frac{\abs{\varphi(n)}}{2^{\abs{I_n}}}<\varepsilon_n}\cap \set{n\in\omega}{\varepsilon_n<1}$. Thus, $w$ is an intersection of two cofinite sets.}.\]
By the assumption, $v$ and $w$ are cofinite sets, which implies that $v\cap w$ is cofinite. Later, for each $n\in v\cap w$ find $X_n$ such that
$$\varepsilon_n\leq \frac{\abs{\varphi(n)\cup X_n}}{2^{\abs{I_n}}}<\varepsilon_n+2^{-\abs{I_n}}.$$
Note that there always exists such an $X_n$. Indeed, since $\frac{\abs{\varphi(n)}}{2^{\abs{I_n}}}<\varepsilon_n<1$, there exist some functions $x\in {^{I_n}2}\smallsetminus \varphi(n)$. Adding one $x$ to $\varphi(n)$ increases the fraction $\frac{\abs{\varphi(n)}}{2^{\abs{I_n}}}$ by $\frac{1}{2^{\abs{I_n}}}$.

Finally, define $\bar{\varphi}$ as follows: For every $n\in\omega$ let
$$\bar{\varphi}(n)=\begin{cases}
    \varphi(n)\cup X_n, & \text{if } n\in v\cap w,\\
    \emptyset,& \text{otherwise}.
\end{cases}$$
Clearly, $[\bar{\varphi}]_*=\set{x\in {^\omega 2}}{\forall^\infty n \in \omega: x{\restriction} I_n\in \bar{\varphi}(n)}$ is in $\Ewf$. Indeed, for any $n\in\omega$ we have
$$\frac{\abs{\bar{\varphi}(n)}}{2^{\abs{I_n}}}<\varepsilon_n+2^{-\abs{I_n}}\leq 2\cdot\varepsilon_n.$$
On the other hand, $[\bar{\varphi}]_*\notin \Ewf_{I,\varepsilon}$ since there is an infinite set $v\cap w$ such that $\frac{\abs{\bar{\varphi}(n)}}{2^{\abs{I_n}}}\geq \varepsilon_n$ for all $n\in v\cap w$.

 \ref{neqIec}:
 If $\Ewf_{I,\varepsilon}=\{\emptyset\}$ for some $I,\varepsilon$, then the proof is done. So, assume that $\Ewf_{I,\varepsilon}\neq\{\emptyset\}$. Take any infinite and co-infinite $B\subseteq \omega$ and define $\varphi(n)=\{\boldsymbol{0}\}$ on $B$ ($\boldsymbol{0}$ represents a~constant zero sequence on any $I_n$), and $\varphi(n)=\emptyset$ on $\omega\smallsetminus B$. Then clearly $\varphi\in\Sigma_{I,\varepsilon}$ and $\emptyset\neq[\varphi]_\infty\in\Swf_{I,\varepsilon}$. It remains to see that $[\varphi]_\infty\notin\Ewf_{I,\varepsilon}$. To this end, we will show that $[\varphi]_\infty\not\subseteq[\psi]_*$ for any $\psi\in\Sigma_{I,\varepsilon}$. So, let $\psi\in\Sigma_{I,\varepsilon}$. Then ${^{I_n}2}\neq\psi(n)$ for all but finitely many $n\in \omega\smallsetminus B$ by the definition of $\Sigma_{I,\varepsilon}$. Define $v=\set{n\in \omega\smallsetminus B}{{^{I_n}2}\neq\psi(n)}$. Note that $\abs{v}=\aleph_0$. Then define $x\in{^\omega2}$ by
 $$x{\restriction} I_n = \begin{cases}
     \boldsymbol{0},& \text{ if } n\in \omega\smallsetminus v,\\
    \text{any } a\in {^{I_n}2}\smallsetminus \psi(n),& \text{ if } n\in v.
 \end{cases}$$
 Therefore $x\in[\varphi]_\infty$ and $x\notin [\psi]_*$.

\ref{neqIed}:
 We are going to show $\Swf_{I,\varepsilon}\not\subseteq\Ewf$. Choose any $\varphi\in\Sigma_{I,\varepsilon}$ such that $[\varphi]_\infty\neq\emptyset$. We will show that $[\varphi]_\infty\notin\Ewf$, i.e., we will show that for any $J\in\Ior$ and any $\psi\in\prod_{n\in\omega}\Pwf({^{J_n}}2)$ with $\sum_{n\in\omega}\frac{\abs{\psi(n)}}{2^{\abs{J_n}}}<\infty$, we have $[\varphi]_\infty\not\subseteq[\psi]_*$ (see \autoref{a13}). Thus, fix such $J$ and $\psi$.
\begin{itemize}
    \item[] \textbf{Case 1.} If $\exists^\infty n:\psi(n)=\emptyset$, then $[\psi]_*=\emptyset$ and consequently $\emptyset\neq[\varphi]_\infty\not\subseteq [\psi]_*$.
    \item[] \textbf{Case 2.} If $\forall^\infty n:\psi(n)\neq\emptyset$ then $\forall^\infty n:\psi(n)\subsetneq {^{J_n}2}$ by $\sum_{n\in\omega}\frac{\abs{\psi(n)}}{2^{\abs{J_n}}}<\infty$. Define increasing sequences $\seq{k_n}{n\in\omega}$, $\seq{j_n}{n\in\omega}$ such that 
    \begin{itemize}
        \item[$\bullet$] $\bigcup_{n\in\omega}I_{k_n}\cap \bigcup_{n\in\omega}J_{j_n}=\emptyset$, and
    \item[$\bullet$] $\varphi(k_n)$ is nonempty for all $n\in\omega$. \end{itemize}
Consider the set $X=\set{x\in\cantor}{\forall n\in\omega: x{\restriction}I_{k_n+1}\in\varphi(k_n+1)}$. Clearly, $X\subseteq [\varphi]_\infty$. On the other hand, pick $x\in\cantor$ such that 
\begin{itemize}
    \item[$\bullet$] $\forall n: x{\restriction}I_{k_n+1}\in\varphi(k_n+1)$ ($\Rightarrow x\in X$),
    \item[$\bullet$] $\forall n: x{\restriction}J_{j_n}\notin\psi(n)$ ($\Rightarrow x\notin [\psi]_*$).
\end{itemize}
Hence, $X\not\subseteq [\psi]_*$.
\end{itemize}
 \end{proof}

\section{Monotonicity}\label{sec:mon}

In the current section we investigate what requirement between two pairs $(I,\varepsilon)$ and $(I',\varepsilon')$ has to be posed in order to achieve $\Swf_{I,\varepsilon}\subseteq \Swf_{I',\varepsilon'}$. Similarly for $\Ewf_{I,\varepsilon}\subseteq \Ewf_{I',\varepsilon'}$. However, we have to review the terminology on relational systems. 

The standard sources for the terminology on relational systems are~\cite{blass,vo93}. We say that $\Rbf=\la X, Y, \sqsubset\ra$ is a \textit{relational system} if it consists of two non-empty sets $X$ and $Y$, and a~relation~$\sqsubset$.
\begin{enumerate}[label=(\arabic*)]
    \item A set $F\subseteq X$ is \emph{$\Rbf$-bounded} if $(\exists y\in Y) (\forall x\in F)\ x \sqsubset y$. 
    \item A set $D\subseteq Y$ is \emph{$\Rbf$-dominating} if $(\forall x\in X) (\exists y\in D)\ x \sqsubset y$. 
\end{enumerate}

We associate two cardinal characteristics with the relational system $\Rbf$: 
\begin{align*}
\bfrak(\Rbf)&:=\min\set{|F|}{\text{$F\subseteq X$ is $\Rbf$-unbounded}},
&&\text{the \emph{unbounding number of\/ $\Rbf$}, and}\\
\dfrak(\Rbf)&:=\min\set{|D|}{\text{$D\subseteq Y$ is $\Rbf$-dominating}},
&&\text{the \emph{dominating number of\/ $\Rbf$}.}
\end{align*}

The dual of $\Rbf$ is defined by $\Rbf^\perp:=\la Y, X, \sqsubset^\perp\ra$ where $y\sqsubset^\perp x$ iff $x\not\sqsubset y$. Note that $\bfrak(\Rbf^\perp)=\dfrak(\Rbf)$ and $\dfrak(\Rbf^\perp)=\bfrak(\Rbf)$.

The cardinal $\bfrak(\Rbf)$ may be undefined, in which case we write $\bfrak(\Rbf) = \infty$, likewise for $\dfrak(\Rbf)$. Concretely, $\bfrak(\Rbf) = \infty$ iff $\dfrak(\Rbf) =1$; and $\dfrak(\Rbf)= \infty$ iff $\bfrak(\Rbf) =1$.

Cardinal characteristics associated with ideals can be characterized by relational systems:

\begin{example}\label{exm:Iwf}
For $\Iwf\subseteq\pts(X)$, define the relational systems:
\begin{enumerate}[label=\rm(\arabic*)]
    \item $\Iwf:=\la\Iwf,\Iwf,{\subseteq}\ra$, which is a directed partial order when $\Iwf$ is closed under unions (e.g.\ an ideal),
    
    \item $\Cbf_\Iwf:=\la X,\Iwf,{\in}\ra$.
\end{enumerate}
Whenever $\Iwf$ is an ideal on $X$,
\begin{enumerate}[label=\rm(\alph*)]
     \item $\bfrak(\Iwf)=\add(\Iwf)$,

    \item $\dfrak(\Iwf)=\cof(\Iwf)$,

    \item $\dfrak(\Cbf_\Iwf)=\cov(\Iwf)$,

    \item $\bfrak(\Cbf_\Iwf)=\non(\Iwf)$.
    
\end{enumerate}
\end{example}

We shall extensively use a so-called Tukey connection,\footnote{Some authors use ``Galois-Tukey connection".} which is a practical tool to determine relations between cardinal characteristics.

\begin{definition}\label{def:Tukey}
Let $\Rbf=\la X,Y,\sqsubset\ra$ and $\Rbf'=\la X',Y',\sqsubset'\ra$ be two relational systems. We say that 
\[
(\Psi_-,\Psi_+)\colon\Rbf\to\Rbf'
\]
is a \emph{Tukey connection from $\Rbf$ into $\Rbf'$} if 
 $\Psi_-\colon X\to X'$ and $\Psi_+\colon Y'\to Y$ are functions such that
 \[
 (\forall x\in X)(\forall y'\in Y')\ 
 \Psi_-(x) \sqsubset' y' \Rightarrow x \sqsubset \Psi_+(y').
 \]
The \emph{Tukey order} between relational systems is defined by
$\Rbf\leqT\Rbf'$ iff there is a Tukey connection from $\Rbf$ into $\Rbf'$. \emph{Tukey equivalence} is defined by $\Rbf\eqT\Rbf'$ iff $\Rbf\leqT\Rbf'$ and $\Rbf'\leqT\Rbf$.
\end{definition}

The crucial assertion connecting Tukey order and cardinal invariants is the following one.

\begin{fact}\label{cor:Tukeyval}
Let\/ $\Rbf=\la X, Y,\sqsubset\ra$ and\/ $\Rbf'=\la X', Y',\sqsubset^{\prime}\ra$ be relational systems. Then
\begin{enumerate}[label=\rm(\alph*)]
    \item $\Rbf\leqT\Rbf'$ implies\/ $(\Rbf')^\perp\leqT\Rbf^\perp$.
    \item\label{Tukeyval:b}
    $\Rbf\leqT\Rbf'$ implies\/ $\bfrak(\Rbf')\leq\bfrak(\Rbf)$ and\/ $\dfrak(\Rbf)\leq\dfrak(\Rbf')$.
    \item $\Rbf\eqT\Rbf'$ implies\/ $\bfrak(\Rbf')=\bfrak(\Rbf)$ and\/ $\dfrak(\Rbf)=\dfrak(\Rbf')$.
    \qed
\end{enumerate}
\end{fact}

Hence, we can proceed to a~question about the role of the parameters in families $\Swf_{I,\varepsilon}$, i.e., what kind of (natural) relation between partitions or elements of $\ell^1_+$ guarantees the relation (inclusion) between families $\Swf_{I,\varepsilon}$. The same question may be asked for $\Ewf_{I,\varepsilon}$.

\begin{definition}\label{a4}
Define the following relation on $\Ior$:
   \[ I \sqsubseteq J \text{ iff } \forall^\infty\, n<\omega\ \exists\, m<\omega\colon I_m\subseteq J_n.
   \]
Note that $\sqsubseteq$ is a directed preorder on $\Ior$, so we think of $\Ior$ as the relational system with the relation $\sqsubseteq$. Also notice that we can consider the family of all partitions of all cofinite subsets of $\omega$ instead of $\Ior$.\footnote{ This is a simple consequence of the fact that $\Ior$ is a $\sqsubseteq$-cofinal subset in such a family. Hence, its bounding and dominating cardinal characteristics are the same as $\bfrak(\Ior)$ and $\dfrak(\Ior)$, respectively. Thus, most of the time, by $\Ior$ we mean the partition of some cofinite subset of $\omega$.}    In Blass~\cite{blass} it is proved that $\Ior \eqT \baire$.

\end{definition}

 The second relation on $\Ior$ is the refinement relation $\refin{}{}$ defined by $\refin{I}{J}$ for $I,J\in \Ior$ iff $$\forall^\infty n \in \omega\ \exists F\in [\omega]^{<\omega}:J_n=\bigcup_{k\in F}I_k,$$
  i.e., all but finitely many $J_n$'s are finite unions of some $I_k$'s. Note that $\refin{I}{J}$ implies $\forall^\infty k \ \exists n : I_k\subseteq J_n$. Also, if $\refin{I}{J}$, then $I\sqsubseteq J$. For $I\sqsubseteq J$ we define $\subin{I}{J}{n}=\set{k\in\omega}{I_k\subseteq J_n}$.

Unfortunately, as we shall see in \autoref{sqrelation_not_working}, none of these relations is sufficient when it comes to inclusions between the studied $\sigma$-ideals on its own. However, after considering some reasonable additional assumptions they might come in useful. On the other hand, we show that the standard relation $\leq^*$ between elements of $\ell^1_+$ yields the inclusion between the families $\Swf_{I,\varepsilon}$ and $\Ewf_{I,\varepsilon}$ without a need for additional assumptions.
\begin{prop}\label{a3}
\mbox{}
Let  $\varepsilon, \varepsilon'\in\ell^1_+$ and $I, J\in\Ior$.
\begin{enumerate}[label=\normalfont(\alph*)]
    \item\label{a3a} If $\varepsilon\leq^* \varepsilon'$, then $\Swf_{I,\varepsilon}\subseteq \Swf_{I,\varepsilon'}$ and $\Ewf_{I,\varepsilon}\subseteq \Ewf_{I,\varepsilon'}$.
\item\label{a3b} Let $k\geq 1$. If $\varepsilon\leq^*\varepsilon'\leq^*k\cdot \varepsilon$, then $\Swf_{I,\varepsilon}=\Swf_{I,\varepsilon'}$ and $\Ewf_{I,\varepsilon}=\Ewf_{I,\varepsilon'}$.\
\item\label{a3d} If $I\sqsubseteq J$, $\varepsilon$ is decreasing and $\forall^\infty n: \min\subin{I}{J}{n}\geq n$,\footnote{This happens, e.g., when $\exists^\infty n: \abs{\set{k}{J_n\cap I_k\neq \emptyset}}\geq 2$.} then $\Ewf_{I,\varepsilon}\subseteq \Ewf_{J,\varepsilon}$.
    \item\label{a3c} If $\refin{I}{J}$ and $\forall^\infty n: \sum_{k\in \subin{I}{J}{n}}\varepsilon_k\leq \varepsilon_n$, then $\Swf_{I,\varepsilon}\subseteq \Swf_{J,\varepsilon}$.
\end{enumerate}
\end{prop}
\begin{proof}
\noindent\ref{a3a}:  Observe that in such a case we have  $$\forall^\infty n: \frac{\abs{\varphi(n)}}{2^{\abs{I_n}}\cdot \varepsilon'_n}\leq\frac{\abs{\varphi(n)}}{2^{\abs{I_n}}\cdot \varepsilon_n},$$
  therefore $\Sigma_{I,\varepsilon}\subseteq \Sigma_{I,\varepsilon'}$.
  \par

\noindent\ref{a3b}: The ``$\subseteq$'' follows from \ref{a3a}. For the ``$\supseteq$'' part let WLOG $\varphi\in\Sigma_{I,k\cdot\varepsilon}$. We will show that $\varphi\in\Sigma_{I,\varepsilon}$. We have 
$$\lim_{n\to\infty}\frac{\abs{\varphi(n)}}{2^{\abs{I_n}}\cdot k\cdot \varepsilon_n}=\frac{1}{k}\cdot\lim_{n\to\infty}\frac{\abs{\varphi(n)}}{2^{\abs{I_n}}\cdot \varepsilon_n}=0.$$

\noindent\ref{a3d}: Let $\varphi\in\Sigma_{I,\varepsilon}$. For every $n$ define 
$$\psi(n)=\largeset{x\in {^{J_n}2}}{\forall k\in\subin{I}{J}{n}: x{\restriction}I_k\in\varphi(k)}=$$
$$=\largeset{x\in{^{J_n}2}}{\forall k: I_k\subseteq J_n\to x{\restriction}I_k\in\varphi(k) }.$$

Since $\varepsilon\to 0$, we have $\forall^\infty k: \varepsilon_k<1$. Furthermore, $\varepsilon$ is decreasing and so $\prod_{k\in\subin{I}{J}{n}}\varepsilon_k\leq\min\set{\varepsilon_k}{k\in\subin{I}{J}{n}}\leq\varepsilon_n$. Thus for all but finitely many $n\in\omega$ we have
$$\frac{\abs{\psi(n)}}{2^{\abs{J_n}}\cdot\varepsilon_n}=\frac{\prod_{k\in\subin{I}{J}{n}}\abs{\varphi(k)}\cdot 2^{\abs{J_n}-\sum_{j\in\subin{I}{J}{n}}\abs{I_j}}}{2^{\abs{J_n}}\cdot\varepsilon_n}=\frac{1}{\varepsilon_n}\cdot\prod_{k\in\subin{I}{J}{n}}\frac{\abs{\varphi(k)}}{2^{\abs{I_k}}}$$
$$\leq\prod_{k\in\subin{I}{J}{n}}\frac{\abs{\varphi(k)}}{2^{\abs{I_k}}\cdot\varepsilon_k}\to 0.$$
The latter term converges to zero since $\varphi\in\Sigma_{I,\varepsilon}$. Therefore, 
$\psi\in \Sigma_{J,\varepsilon}$. Note that $[\varphi]_*\subseteq [\psi]_*$. Indeed, if $x\in\cantor$ is such that $\forall^\infty k:x{\restriction}I_k\in \varphi(k)$, then $\forall^\infty n: x{\restriction}J_n\in\psi(n)$.

\noindent\ref{a3c}:     We will show that if $\varepsilon$ decreases sufficiently fast, then $\refin{I}{J}$ implies $S_{I,\varepsilon}\subseteq S_{J,\varepsilon}$. So, let $\varepsilon\in\ell^1_+$ be such that $\forall^\infty n: \sum_{k\in \subin{I}{J}{n}}\varepsilon_k\leq \varepsilon_n$. For any $\varphi\in \Sigma_{I,\varepsilon}$ set $$\psi(n)=\set{s\in {^{J_n}2}}{\exists k\in \subin{I}{J}{n}: s{\restriction} I_k\in \varphi(k)},$$
  i.e., $\psi(n)$ is a set of extensions of functions from $\bigcup_{k\in\subin{I}{J}{n}}\varphi(k)$ on $J_n$. 
  
  First, note that for any $N\in\omega$ we have that for all but finitely many $n\in\omega$, $$\frac{\abs{\psi(n)}}{2^{\abs{J_n}}}=\sum_{k\in\subin{I}{J}{n}}\frac{\abs{\varphi(k)}\cdot 2^{\abs{J_n\smallsetminus I_k}}}{2^{\abs{I_k}}\cdot 2^{\abs{J_n\smallsetminus I_k}}}=\sum_{k\in\subin{I}{J}{n}}\frac{\abs{\varphi(k)}}{2^{\abs{I_k}}}<\frac{\sum_{k\in\subin{I}{J}{n}}\varepsilon_k}{N}\leq \frac{\varepsilon_n}{N},$$
  thus, $\psi\in \Sigma_{J, \varepsilon}$.
  
  Notice that $[\varphi]_\infty\subseteq [\psi]_\infty$. Indeed, if $x\in {^\omega 2}$ is such that $\exists^\infty k:x{\restriction} I_k\in \varphi(k)$. Consider $\seq{n_k}{k\in\omega}$ such that $I_k\subseteq J_{n_k}$ for all but finitely many $k\in\omega$ (there is such a sequence by $\refin{I}{J}$). Then $x{\restriction} J_{n_k}$ is an extension of $x{\restriction} I_k\in \varphi(k)$ on $J_{n_k}$, that is, $x{\restriction} J_{n_k}\in \psi(n_k)$ for all but finitely many $k\in\omega$. Therefore, we get $\exists^\infty n: x{\restriction} J_n\in \psi(n)$.

  Since for any $\varphi\in \Sigma_{I,\varepsilon}$ we found $\psi\in \Sigma_{J,\varepsilon}$ with $[\varphi]_\infty\subseteq [\psi]_\infty$, we have $\Swf_{I,\varepsilon}\subseteq \Swf_{J,\varepsilon}$.
\end{proof}

    In \autoref{a3}, \ref{a3d} - \ref{a3c}, the relations $\sqsubseteq$ and $\sqsubseteq_R$ are not sufficient to guarantee the corresponding conclusions.    
    
    \begin{observation}\label{sqrelation_not_working}
    There are $I, J, \varepsilon$ such that $I\sqsubseteq J$ and $\Swf_{I,\varepsilon}\not\subseteq \Swf_{J,\varepsilon}$. The same is true for the relation $\sqsubseteq$ with $\Ewf_{I,\varepsilon}$, and for the relation $\sqsubseteq_R$ with both, $\Swf_{I,\varepsilon}$ and $\Ewf_{I,\varepsilon}$.
    \end{observation}

     \begin{proof}
     The assertions for the relation $\sqsubseteq_R$ can be proven similarly to the assertions for the relation $\sqsubseteq$. Thus we prove just the latter one. 
     
We shall show first that given any directed preorder $\preceq$ on $\Ior$, it is not necessarily true that $I\preceq J$ implies $\Swf_{I,\varepsilon}\subseteq \Swf_{J,\varepsilon}$ for any $\varepsilon\in\ell^1_+$. Assume that the implication holds. By the proof of~\cite[Lem.~1.3]{bartosmall}, there are $X,Y\in \Swf^\star$ such that $X\cup Y\notin \Swf^\star$. By~\autoref{a2}~\ref{a2b}, choose $I,I'\in\Ior$ and $\varepsilon,\varepsilon'\in \ell^1_+$ such that $X\in \Swf_{I,\varepsilon}$ and  $Y\in\Swf_{I',\varepsilon'}$. Define $\delta\in \ell^1_+$ by $\delta_n=\max\{\varepsilon_n,\varepsilon'_n\}$ for each $n$. Then $\Sigma_{I,\varepsilon}\subseteq \Sigma_{I,\delta}$, $\Sigma_{I',\varepsilon'}\subseteq \Sigma_{I',\delta}$, and $X\in \Swf_{I,\delta}$, $Y\in \Swf_{I',\delta}$. 

On the other hand, since $\preceq$ is a directed preorder, we can find a partition $J$ such that $I, I'\preceq J$. By the assumption, both $\Swf_{I,\delta}, \Swf_{I',\delta}\subseteq \Swf_{J,\delta}$, so we get $X, Y\in \Swf_{J,\delta}$. Therefore, $X\cup Y\in \Swf_{J,\delta}\subseteq \Swf^\star$, which is a contradiction. 
  
Regarding $\Ewf_{I,\varepsilon}$, we shall show that $I\sqsubseteq J$ does not necessarily imply $\Ewf_{I,\varepsilon}\subseteq \Ewf_{J,\varepsilon}$. That is, there are $I,J\in\Ior$ with $I\sqsubseteq J$ and $\varepsilon\in\ell^1_+$ such that $(I,\varepsilon)$ is $\Ewf$-contributive and $J$ is not $\Ewf$-contributive (even $\Swf^\star$-contributive). Thus, $\{\emptyset\}\neq\Ewf_{I,\varepsilon}\not\subseteq \Ewf_{J,\varepsilon}=\{\emptyset\}$. To show this, construct inductively $I$ such that $$\forall^\infty n: 2^{-\abs{I_n}}<\min\left\{\frac{2^{-\abs{I_{n-1}}}}{n}, 2^{-(n+1)}\right\}.$$ It is easy to construct such a partition -- each step takes sufficiently long $I_n$. Put $\varepsilon_n=2^{-\abs{I_{n-1}}}$ for $n\geq 2$. Let $I_0$ be an arbitrary finite interval (initial in $\omega$) of length at least $2$, $\abs{I_1}=2$, $\varepsilon_0=1$ and $\varepsilon_1=\frac{1}{2}$. Note that $\varepsilon_n<2^{-n}$ for all but finitely many $n\in\omega$, thus, $\varepsilon\in\ell^1_+$. Also, note that $\forall^\infty n: 2^{-\abs{I_n}}<\frac{\varepsilon_n}{n}$, therefore, by \autoref{charofrEIe} we have that $(I,\varepsilon)$ is $\Ewf$-contributive. Define $J$ such that $I_0=J_0\cup J_1$ and $J_n=I_{n-1}$ for all $n\in\omega$. Clearly, $I\sqsubseteq J$ (even $J\sqsubseteq I$). Moreover, $\forall^\infty n: 2^{-\abs{J_n}}=2^{-\abs{I_{n-1}}}=\varepsilon_n$. Thus, by \autoref{charofrEIe} and \autoref{charofrSIe} we have that $\Ewf_{J,\varepsilon}=\Swf_{J,\varepsilon}=\{\emptyset\}$. 
\end{proof}

The refinement relation has minimal usefulness regarding cardinal invariants since its bounding and dominating numbers are trivial. Denote by $\mathbf{D}_\mathrm{R}$ the relational system $\la \Ior,\Ior, \sqsubseteq_\mathrm{R} \ra$.

\begin{prop}\label{bad_rel}
$\bfrak(\mathbf{D}_\mathrm{R})=2$ and $\dfrak(\mathbf{D}_\mathrm{R})=\cfrak$.
\end{prop}
\begin{proof}
First, it is very easy to see that the relation $\sqsubseteq_\mathrm{R}$ is not a directed preorder. To see this, let $f\in {^\omega\omega}$ be an increasing function. Put $I_n=[f(2n), f(2n+2))$, $J_0=[f(0),f(3))$, and for $n\geq 1$ put $J_n=[f(2n+1), f(2n+3))$. We get overlapping partitions for which we cannot find a common upper bound. Hence, $\bfrak(\mathbf{D}_\mathrm{R})=2$.

We shall formulate previous facts as a claim since we use it in the next construction. First, for any $I=\set{[i_k, i_{k+1})}{k\in\omega}\in\Ior$ define the set of endpoints of its intervals $\ep{I}=\set{i_k}{k\in\omega}$.

  \begin{claim}\label{overlap}
      If $ I, J\in\Ior$ satisfy $\abs{\ep{I}\cap \ep{J}}<\omega$, then no $K\in\Ior$ is an $\sqsubseteq_\mathrm{R}$-upper bound for both $I$ and $J$, i.e., for any $K^I,K^J\in\Ior$ such that $\refin{I}{K^I}$ and $\refin{J}{K^J}$ we have $K^I\neq K^J$.
  \end{claim}
\begin{proof}
Assume not. Let $K\in\Ior$ be such that $\refin{I,J}{K}$. Then we have
$$\exists m_0\in\omega\ \forall m>m_0: K_m\text{ is a union of some }I_k\text{'s},$$
$$\exists m_1\in\omega\ \forall m>m_1: K_m\text{ is a union of some }J_n\text{'s}.$$

By $\abs{\mathrm{ep}(I)\cap \mathrm{ep}(J)}<\omega$ we have also
$$\exists M\in\omega\ \forall k,n>M:$$
 $$  \min(I_k)\neq \min(J_n), \max(J_n) \text{ and } \max(I_k)\neq  \min(J_n), \max(J_n).$$
Hence, 
$$\exists m_2\in\omega\ \forall m>m_2:\min\subin{I}{K}{m}>M\text{ and }\min\subin{J}{K}{m}>M.$$

Now consider any $m>\max\{m_0, m_1, m_2\}$. For this $m$, $K_m$ is a union of $J_n$'s for $n\in\subin{J}{K}{m}$. Since $m>m_0$, $K_m$ is also a union of $I_k$'s for $k\in \subin{I}{K}{m}$. But this is not possible by the fact that $m>m_2$, which implies that for any $k\in \subin{I}{K}{m}$ and $n\in\subin{J}{K}{m}$, $I_k$ and $J_n$ do not share common endpoints.  
\end{proof}

We will show that $\dfrak(\sqsubseteq_\mathrm{R})=\cfrak$. Let $\Dwf$ be a $\sqsubseteq_\mathrm{R}$-dominating family in $\Ior$. Consider an arbitrary AD family $\Awf$ of cardinality $\cfrak$. Assign to any $A=\set{a_i\in\omega}{i\in\omega}\in\Awf$ (naturally enumerated in an increasing way) a partition $I_A=\set{[a_i, a_{i+1})}{i\in\omega}$ and denote $\Iwf_\Awf=\set{I_A}{A\in \Awf}$. Note that for any $A\neq B\in \Awf$, $I_A$ and $I_B$ satisfy $\abs{\ep{I_A}\cap \ep{I_B}}<\infty$. By \autoref{overlap}, no $K\in\Ior$ is bound for two distinct elements of $\Iwf_\Awf$, let alone infinitely many of them. That is, for any $A\in \Awf$ there is $K_A\in \Dwf$, a bound for $I_A$, and by \autoref{overlap}, this $K_A$ is an upper bound for no other $I_B$, $B\neq A\in \Awf$. Therefore $\abs{\set{K_A}{A\in\Awf}}=\cfrak$ and since $\set{K_A}{A\in\Awf}\subseteq \Dwf$, $\abs{\Dwf}=\cfrak$ as well.
\end{proof}

We shall establish a basic tool for showing inclusions between the studied $\sigma$-ideals, i.e., we show that the inclusion strongly depends on coordinate-wise inclusions.
\begin{lemma}\label{subsetpointwiseS}
        Let $I\in\Ior$, $\varepsilon\in\ell^1_+$, and $\varphi, \psi\in \Sigma_{I,\varepsilon}$. The following statements are equivalent.
        \begin{enumerate}[label=\rm(\arabic*)]
            \item\label{subsetpointwise1} $\forall^\infty n\in\omega: \varphi(n)\subseteq \psi(n)$.
            \item\label{subsetpointwise2} $[\varphi]_\infty\subseteq [\psi]_\infty$.
        \end{enumerate}
    \end{lemma}
    \begin{proof}
       \ref{subsetpointwise1}~$\Rightarrow$~\ref{subsetpointwise2} is clear: $x\in[\varphi]_\infty$ iff $\exists^\infty n: x{\restriction} I_n\in\varphi(n)$. Since $\forall^\infty n: \varphi(n)\subseteq\psi(n)$ we have that $\exists^\infty n: x{\restriction} I_n\in\psi(n)$.

        $\neg$\ref{subsetpointwise1}~$\Rightarrow\neg$\ref{subsetpointwise2}: By the assumption there exists an increasing sequence $\seq{k_n}{n\in\omega}$ such that $\varphi(k_n)\smallsetminus \psi(k_n)\neq \emptyset$. Consider any $x\in {^\omega2}$ such that $x{\restriction} I_{k_n} \in\varphi(k_n)\smallsetminus \psi(k_n)$ for all $n\in\omega$, and $x{\restriction} I_m\notin \psi(m)$ for all but finitely many $m\in \omega\smallsetminus \set{k_n}{n\in\omega}$. Then clearly $x\in [\varphi]_\infty$ but $x\notin [\psi]_\infty$.
    \end{proof}

\begin{remark}
    Note that for any $\varphi\in\Sigma_{I,\varepsilon}$ we have $[\varphi]_*\neq\emptyset$ if and only if $\forall^\infty n\in\omega: \varphi(n)\neq\emptyset$.
\end{remark}

      \begin{lemma}\label{subsetpointwiseE}
        Let $I\in\Ior$, $\varepsilon\in\ell^1_+$, and let $\varphi, \psi\in \Sigma_{I,\varepsilon}$ be such that $\forall^\infty n\in\omega: \varphi(n)\neq\emptyset$. The following statements are equivalent.
        \begin{enumerate}[label=\rm(\arabic*)]
            \item\label{subsetpointwiseE1} $\forall^\infty n\in\omega: \varphi(n)\subseteq \psi(n)$.
            \item\label{subsetpointwiseE2} $[\varphi]_*\subseteq [\psi]_*$.
            \item\label{subsetpointwiseE3} $[\varphi]_\infty\subseteq [\psi]_\infty$.
        \end{enumerate}
    \end{lemma}
    \begin{proof}
      \ref{subsetpointwise1}~$\Rightarrow$~\ref{subsetpointwiseE2} is clear: $x\in[\varphi]_*$ iff $\forall^\infty n: x{\restriction} I_n\in\varphi(n)$. Again, using the fact that $\forall^\infty n: \varphi(n)\subseteq\psi(n)$ we have that $\forall^\infty n: x{\restriction} I_n\in\psi(n)$.

       $\neg$\ref{subsetpointwise1}~$\Rightarrow\neg$\ref{subsetpointwiseE2}: By the assumption $\forall^\infty n: \varphi(n)\neq\emptyset$ and there exists an increasing sequence $\seq{k_n}{n\in\omega}$ such that $\varphi(k_n)\smallsetminus \psi(k_n)\neq \emptyset$. Consider any $x\in {^\omega2}$ such that $x{\restriction} I_{k_n} \in\varphi(k_n)\smallsetminus \psi(k_n)$ for all $n\in\omega$, and $x{\restriction} I_m\in \varphi(m)$ for all but finitely many $m\in \omega\smallsetminus \set{k_n}{n\in\omega}$. Then clearly $x\in [\varphi]_*$ but $x\notin [\psi]_*$.

       \ref{subsetpointwiseE3} is equivalent with \ref{subsetpointwiseE1} by \autoref{subsetpointwiseS}.
    \end{proof}

 \begin{remark}
    Note that for the implication \ref{subsetpointwiseE1}~$\Rightarrow$~\ref{subsetpointwiseE2} in \autoref{subsetpointwiseE}, we do not need to assume that $\varphi(n)\neq\emptyset$ for all but finitely many $n$'s.
\end{remark}

With respect to \autoref{a3}, parts \ref{a3a} and \ref{a3b}, we shall demonstrate that for distinct $\varepsilon,\varepsilon'$ we may have distinct $\Swf_{I,\varepsilon}, \Swf_{I,\varepsilon'}$. Similarly for $\Ewf_{I,\varepsilon}$. That is, it is not necessarily true that $\Swf_{I,\varepsilon}=\Swf_{I,\varepsilon'}$ and $\Ewf_{I,\varepsilon}=\Ewf_{I,\varepsilon'}$ for $\varepsilon\neq^*\varepsilon'\in\ell^1_+$. 
\par
The above fact can be easily proven by considering $I,\varepsilon,\varepsilon'$ such that $(I,\varepsilon)$ is not $\Swf^\star$-contributive (subsequently nor $\Ewf$-contributive) but $(I,\varepsilon')$ is $\Ewf$-contributive (subsequently also $\Swf^\star$-contributive). Indeed, consider $I$ such that $\abs{I_n}=n$ and define $\varepsilon,\varepsilon'$ by $\varepsilon_n=\frac{1}{2^n}$ and $\varepsilon_n'=\frac{n}{2^n}$ for every $n\in\omega$. Then for any $A\in[\omega]^\omega$ we have $\forall^\infty n\in A: 2^{\abs{I_n}}\cdot \varepsilon_n= \frac{2^n}{2^n}=1$, thus, $\lim_{n\in A}2^{\abs{I_n}}\cdot \varepsilon_n\neq \infty$. By \autoref{charofrSIe}, $(I,\varepsilon)$ is not $\Swf^\star$-contributive. On the other hand, $\lim_{n\to\infty} 2^{\abs{I_n}}\cdot\varepsilon_n'=\frac{2^{\abs{I_n}}\cdot n}{2^{\abs{I_n}}}=n\to \infty$. By \autoref{charofrEIe} we have that $(I,\varepsilon')$ is $\Ewf$-contributive. 
\par
The question is, whether $\Swf_{I,\varepsilon}$ or $\Ewf_{I,\varepsilon}$, once they are nontrivial, can differ w.r.t. $\varepsilon$. 
We will answer this question in the following example.

\begin{observation}
    $\Swf_{I,\varepsilon}\neq\Swf_{I,\varepsilon'}$ and $\Ewf_{I,\varepsilon}\neq\Ewf_{I,\varepsilon'}$ for any $I$ with $\abs{I_n}\geq (n+1)^3$ and $\varepsilon,\varepsilon'$ defined by
    $$\varepsilon_n=\frac{1}{(n+1)^{2}}+\frac{n}{2^n} \hspace{1cm}\text{and}\hspace{1cm} \varepsilon_n'=\frac{1}{(n+1)^{3}}.$$ 
\end{observation}
\begin{proof}
The proof for $\Ewf_{I,\varepsilon}$ is similar to the proof for $\Swf_{I,\varepsilon}$. Thus we shall prove the assertion just for $\Swf_{I,\varepsilon}$.

 Note that both $(I,\varepsilon)$ and $(I,\varepsilon')$ are $\Ewf$-contributive. 
     Let
     $$\xi_n=\max\largeset{m\in\omega}{\frac{m}{2^{\abs{I_n}}}<\frac{1}{(n+1)^2}}.$$
    \noindent \begin{claim}
        $\forall K\in\omega\ \exists n_K\in\omega\ \forall n>n_K:\frac{\xi_n}{2^{\abs{I_n}}\cdot K}\geq\frac{1}{(n+1)^3}$.
    \end{claim}
    \begin{proof}
        Notice that for all $n\in\omega$ we have 
        $$\frac{\xi_n}{2^{\abs{I_n}}}\geq\frac{1}{(n+1)^2}-\frac{1}{2^{\abs{I_n}}}\geq \frac{1}{(n+1)^2}-\frac{1}{2^n}.$$
        It is easy to show that $\forall K\in\omega\ \forall^\infty n\in\omega: \frac{1}{(n+1)^2}-\frac{1}{2^n}\geq \frac{K}{(n+1)^3}$, since for positive $n$, this inequality is equivalent with
        $$2^n(n-K+1)\geq (n+1)^3,$$
        and $2^n$ increases faster than $(n+1)^3$.
    \end{proof}
We may assume that the sequence $\seq{n_K}{K\in\omega}$ is increasing. Define $\varphi$ such that $\abs{\varphi(i)}=\lceil{\frac{\xi_i}{K}}\rceil$ for $i\in [n_K,n_{K+1})$. Note that $\forall^\infty n: \varphi(n)\neq\emptyset$. We will show that $\varphi\in\Sigma_{I,\varepsilon}$. 

Let $N\in\omega$. Then for all but finitely many $K>N$ and for all $i\in [n_K,n_{K+1})$,
$$\frac{\abs{\varphi(i)}}{2^{\abs{I_i}}}=\frac{\lceil\frac{\xi_i}{K}\rceil}{2^{\abs{I_i}}}\leq \frac{\frac{\xi_i}{K}+1}{2^{\abs{I_i}}}=\frac{\xi_i}{2^{\abs{I_i}}\cdot K}+\frac{1}{2^{\abs{I_i}}}\leq \frac{1}{(i+1)^2\cdot K}+\frac{i}{2^i\cdot K}<\frac{1}{(i+1)^2\cdot N}+\frac{i}{2^i\cdot N}.$$
Thus, by the remark below \autoref{a1}, $\varphi\in\Sigma_{I,\varepsilon}$.
\par
On the other hand, for any $K\in\omega$ and for any $i\in[n_K,n_{K+1})$ we have
$$\frac{\abs{\varphi(i)}}{2^{\abs{I_i}}}=\frac{\lceil\frac{\xi_i}{K}\rceil}{2^{\abs{I_i}}}\geq\frac{\xi_i}{2^{\abs{I_i}}\cdot K}\geq\frac{1}{(i+1)^3}.$$
This yields $[\varphi]_\infty\not\in\Swf_{I,\varepsilon'}$, since otherwise there is $\psi\in\Sigma_{I,\varepsilon'}$ such that $[\varphi]_\infty\subseteq[\psi]_\infty$. However, by \autoref{subsetpointwiseS} we have $\varphi(n)\subseteq\psi(n)$ for all but finitely many $n$'s, and we obtain
$$\frac{\abs{\psi(i)}}{2^{\abs{I_i}}}\geq\frac{\abs{\varphi(i)}}{2^{\abs{I_i}}}\geq\frac{1}{(i+1)^3}=\varepsilon_i.$$
\end{proof}

\section{Additivity and cofinality of the new ideals}\label{sec:zfc}

\newcommand{\twosmall}{2-small$^\star$~coding}

In this section, we will study the cardinals $\add(\Ewf_{I,\varepsilon})$ and $\add(\Swf_{I,\varepsilon})$. We will be mostly concerned with proving~\autoref{ThmM:ac}, \autoref{a12}, and~\autoref{Thm:f}~\ref{Thm:f:a}. 

We begin with the following combinatorial lemma that will be used to prove~\autoref{a12}. 

\begin{lemma}\label{nonemptyphi}
    Let $(I,\varepsilon)\in \Ior\times\ell^1_+$ be an $\Ewf$-contributive pair. Then for any $\varphi\in \Sigma_{I,\varepsilon}$ there is $\bar{\varphi}\in\Sigma_{I,\varepsilon}$ such that $\varphi(n)\subseteq \bar{\varphi}(n)$ and $ \bar{\varphi}(n)\neq\emptyset$ for all but finitely many $n\in\omega$.
\end{lemma}
\begin{proof}
    If $\varphi(n)=\emptyset$ take $\bar{\varphi}(n)=\{\boldsymbol{0}\}$. Otherwise put $\bar{\varphi}(n)=\varphi(n)$. We will show that $\bar{\varphi}\in\Sigma_{I,\varepsilon}$. Putting $v=\set{n\in\omega}{\varphi(n)=\emptyset}$ we have
    \begin{align*}
       & \forall n\in \omega\smallsetminus v: \frac{\abs{\bar{\varphi}(n)}}{2^{\abs{I_n}}}=\frac{\abs{\varphi(n)}}{2^{\abs{I_n}}},\text{ and} \\
       & \forall n\in v: \frac{\abs{\bar{\varphi}(n)}}{2^{\abs{I_n}}}=\frac{1}{2^{\abs{I_n}}}. 
    \end{align*}
   Let $N\in\omega$. By the fact that $\varphi\in\Sigma_{I,\varepsilon}$ and by \autoref{charofrEIe} we have 
   \begin{align*}
      & \forall^\infty n\in \omega\smallsetminus v: \frac{\abs{\bar{\varphi}(n)}}{2^{\abs{I_n}}}<\frac{\varepsilon_n}{N},\text{ and}\\
    &  \forall^\infty n\in v: \frac{\abs{\bar{\varphi}(n)}}{2^{\abs{I_n}}}=\frac{1}{2^{\abs{I_n}}}<\frac{\varepsilon_n}{N}.
   \end{align*}
    Consequently, $\forall^\infty n\in\omega: \frac{\abs{\bar{\varphi}(n)}}{2^{\abs{I_n}}}<\frac{\varepsilon_n}{N}$. 
    \end{proof}

Note that we may assume that the smallest base of $\Ewf_{I,\varepsilon}\neq \{\emptyset\}$ consists only of sets $[\varphi]_*$ such that $\forall^\infty n: \varphi(n)\neq \emptyset$.

    \begin{prop}
    \label{SEcomparison}
        For any $\Ewf$-contributive $I\in\Ior$, $\varepsilon\in\ell^1_+$ we have 
       $\Swf_{I,\varepsilon}\eqT \Ewf_{I,\varepsilon}$.
         In particular, $\cof(\Swf_{I,\varepsilon})=\cof(\Ewf_{I,\varepsilon})$ and      $\add(\Ewf_{I,\varepsilon})=\add(\Swf_{I,\varepsilon}).$
    \end{prop}
    \begin{proof}
$``\leqT":$ For any $X\in\Swf_{I,\varepsilon}$ there is $\varphi_X\in\Sigma_{I,\varepsilon}$ such that $X\subseteq [\varphi_X]_\infty$ and $\forall^\infty n: \varphi_X(n)\neq \emptyset$, see \autoref{nonemptyphi}. Define $\Psi_-\colon \Swf_{I,\varepsilon}\to\Ewf_{I,\varepsilon}$ by $\Psi_-(X)=[\varphi_X]_*$. For any $Y\in\Ewf_{I,\varepsilon}$ there is $\psi_Y\in\Sigma_{I,\varepsilon}$ such that $Y\subseteq[\psi_Y]_*$. Define $\Psi_+\colon\Ewf_{I,\varepsilon}\to\Swf_{I,\varepsilon}$ by $\Psi_+(Y)=[\psi_Y]_\infty$. 

Let $X\in\Swf_{I,\varepsilon}$ and $Y\in\Ewf_{I,\varepsilon}$ be such that $\Psi_-(X)=[\varphi_X]_*\subseteq Y$. Since $Y\subseteq [\psi_Y]_*$, we have $[\varphi_X]_*\subseteq [\psi_Y]_*$. By \autoref{subsetpointwiseE} we have $X\subseteq [\varphi_X]_\infty\subseteq[\psi_Y]_\infty=\Psi_+(Y)$.

$``\geqT":$ To see this just switch the roles of $\Psi_-$ and $\Psi_+$ from the first part of this proof.
    \end{proof}

It is well-know that $\add(\Nwf)$ and $\cof(\Nwf)$ can be characterized using slaloms. We shall show the connection between localization and anti-localization cardinals based on slaloms and the studied $\sigma$-ideals, but first, we need the basic terminology from localization and anti-localization theory. For more details and further references, see \cite{CM23}.

\begin{definition}\label{defloc}
Given a sequence of non-empty sets $b = \seq{b(n)}{n\in\omega}$ and $h\colon \omega\to\omega$, define 
\begin{align*}
 \prod b &:= \prod_{n\in\omega}b(n),  \\
 \Swf(b,h) &:= \prod_{n\in\omega} [b(n)]^{\leq h(n)}.
\end{align*}
For two functions $x\in\prod b$ and $\varphi\in\Swf(b,h)$ write  
\[x\,\in^*\varphi \textrm{\ iff\ }\forall^\infty n\in\omega:x(n)\in \varphi(n),\] and 
\[x\,\in^\infty\varphi\textrm{\ iff\ }\exists^\infty n\in\omega:x(n)\in \varphi(n).\]
The negations of $x\in^*\varphi$ and $x\in^\infty \varphi$ are denoted by $x\notin^*\varphi$ and $x\notin^\infty\varphi$ respectively, i.e., $x\notin^*\varphi$ iff $\exists^\infty n\in\omega: x(n)\notin \varphi(n)$ and $x\notin^\infty \varphi$ iff $\forall^\infty n\in\omega: x(n)\notin \varphi(n)$.

We set
\begin{align*}
 \blc_{b,h}&=\min\set{|F|}{F\subseteq \prod b\;\&\;\neg\exists \varphi \in \Swf(b,h)\,\forall x \in F:x\in^* \varphi},  \\
 \dlc_{b,h}&=\min\set{|R|}{R\subseteq \Swf(b,h)\;\&\;\forall x \in \prod b\,\exists \varphi\in R:x\in^* \varphi},\\
 \balc_{b,h}&=\min\set{|S|}{S\subseteq \Swf(b,h)\;\&\;\neg\exists x \in \prod b\,\forall \varphi\in S\,:x\notin^\infty \varphi},\\
 \dalc_{b,h}&=\min\set{|E|}{E\subseteq \prod b\;\&\;\forall \varphi \in \Swf(b,h)\,\exists x \in E\,:x\notin^\infty \varphi}.
\end{align*}
Denote $\Lc(b,h) := \la\prod b,\Scal(b,h),\in^*\ra$ and $\aLc(b,h) := \la\Scal(b,h),\prod b,\not\ni^\infty\ra$. So we have $\blc_{b,h}=\bfrak(\Lc(b,h))$, $\balc_{b,h}=\bfrak(\aLc(b,h))$, and similarly for the dominating numbers.
\end{definition}

When $b$ is the constant sequence $\omega$, we use the notation $\Lc(\omega,h)$, $\aLc(\omega,h)$ and we denote the associated cardinal characteristics by $\blc_{\omega,h}$, $\dlc_{\omega,h}$, $\balc_{\omega,h}$, $\dalc_{\omega,h}$.

\begin{theorem}[Bartoszyński~\cite{b}, see also~\cite{CM23}]\label{addbloc}
If $h\in {^\omega\omega}$ and $\lim_{n\to\infty}h(n)=\infty$. Then $\blc_{\omega,h}=\add(\Nwf)$ and $\dlc_{\omega,h}=\cof(\Nwf)$.
\end{theorem}

\begin{theorem}[Bartoszyński~\cite{bar}, Miller~\cite{Mil82}]\label{nonbaloc}
If $h\in\baire$ and $h\geq^*1$ then $\balc_{\omega,h}=\non(\Mwf)$ and $\dalc_{\omega,h}=\cov(\Mwf)$.
\end{theorem}

Before we state the relationships between the studied $\sigma$-ideals and localization cardinals, we need to expand the terminology concerning relational systems.
Let $\Rbf=\langle X,Y,\sqsubset \rangle$, $\Rbf'=\langle X',Y',\sqsubset' \rangle$ be relational systems. We define the \textit{sequential composition} $(\Rbf;\Rbf')=\langle X\times (X')^{Y},Y\times Y',\sqsubset^{\bullet}\rangle$, where the binary relation $(x,f)\sqsubset^{\bullet}(a,b)$ means $x\sqsubset a$ and $f(a)\sqsubset^{\prime} b$. More about this way of creating a relational system from existing relational systems can be found in \cite{blass}. Also, we will use the following abstraction of certain types of sequences, allowing us to think about complex objects as natural numbers. Here, we let $\baireinc=\set{f\in\baire}{f\text{ is increasing}}$.

\begin{remark}\label{bairecoding}
Let $I\in\Ior$. For any $a\in\baireinc$, the set $\prod_{n\in\omega}\left(\prod_{m\in[a(n),a(n+1))} {\Pwf(^{I_m}2})\right)$ can be thought of as a subset of ${^\omega\omega}$. 
Indeed, denote $\lambda_i=\abs{\Pwf\left({^{I_i }2} \right)}<\omega$. Then for each $a\in\baireinc$ and $n\in\omega$ there are $\kappa_{a,n}:=\prod_{i=a(n)}^{a(n+1)-1}\lambda_i$ finite sequences $t$ having the domain $[a(n),a(n+1))$ such that $t(i)\subseteq {^{I_i}2}$ for each $i\in [a(n),a(n+1))$. Enumerate these sequences in an arbitrary way to get $\largeset{t^{a,n}_K}{K<\kappa_{a,n}}$. We define $\Phi\colon\bigcup_{a\in\baireinc}\prod_{n\in\omega}\left(\prod_{m\in[a(n),a(n+1))} \Pwf({^{I_m}2})\right) \to \baire$ as follows. For $f\in \bigcup_{a\in\baireinc}\prod_{n\in\omega}\left(\prod_{m\in[a(n),a(n+1))} \Pwf({^{I_m}2})\right)$ there is a unique $a'$ such that $f\in \prod_{n\in\omega}\left(\prod_{m\in[a'(n),a'(n+1))} \Pwf({^{I_m}2})\right)$. We let $\Phi(f)(n)=K$ iff $f(n)=t^{a',n}_K$, i.e., if $f(n)$ is the $K$-th element with respect to the above-mentioned enumeration.\footnote{Notice that this can be done systematically. E.g., enumerate the sets $\Pwf\left({^{I_n}2}\right)=\set{Y_i}{i<\lambda_n}$. Assign to every sequence $S$ having the domain $[a(n), a(n+1))$ such that $t(i)\subseteq {^{I_i}2}$ for each $i\in [a(n),a(n+1))$, the vector $\boldsymbol{c}=\langle c_0, \dotso, c_{a(n+1)-a(n)-1}\rangle$, where the terms $c_i$ are such that $t(i)=Y_{c_i}$ for each $i<a(n+1)-a(n)-1$. Then assign $K$ to $f(n)$ iff the vector $\boldsymbol{c}$ corresponding to $f(n)$ is the $K$-th with respect to the lexicographical ordering.}
\end{remark}

The proof of the following result is based on \cite[Thm.~5.1.2]{bartjud}:

\begin{theorem}\label{hardtukey}
   Let $h\in{^\omega\omega}$ be defined by $h(n)=(n+1)^2$ for all $n\in\omega$. For any $I\in\Ior$ and $\varepsilon\in\ell^1_+$ the following holds.
  \begin{enumerate}[label=\normalfont(\alph*)]
        \item\label{addcofa} If $(I,\varepsilon)$ is $\Ewf$-contributive, then $ \Ewf_{I,\varepsilon}\leqT(\langle\baireinc,\baireinc,\leq^{*}\rangle;\Lc(\omega,h))$,
        \item\label{addcofb} If $(I,\varepsilon)$ is $\Swf$-contributive, then $ \Swf_{I,\varepsilon}\leqT(\langle\baireinc,\baireinc,\leq^{*}\rangle;\Lc(\omega,h))$.
    \end{enumerate}
\end{theorem}
\begin{proof}
\ref{addcofa}: We need to find $\Psi_-\colon\Ewf_{I,\varepsilon}\to\baireinc\times{}^{\baireinc}(\baire)$ and $\Psi_+\colon\baireinc\times\Swf(\omega,h)\to\Ewf_{I,\varepsilon}$. First, we are going to define $\Psi_-$. For $X\in \Ewf_{I,\varepsilon}$ there is $\varphi_X\in\Sigma_{I,\varepsilon}$ such that $X\subseteq [\varphi_X]_*$. Define $k^X=\seq{k^X_n}{n\in\omega}$ by $k^X_0=0$ and for $n\geq 1$
$$k_{n}^X=\min\largeset{m>k^X_{n-1}}{\forall j\geq m: \frac{\abs{\varphi_X(j)}}{2^{\abs{I_j}}}<\frac{\varepsilon_j}{(n+1)^3}}.$$ 
Notice that the existence of $k^X$ is guaranteed by the fact that $\varphi_X\in\Sigma_{I,\varepsilon}$.

We will define a function $F^X\colon \baireinc\to\bigcup_{a\in\baireinc}\prod_{n\in\omega}\left(\prod_{m\in[a(n),a(n+1))} {\Pwf(^{I_m}2)}\right)$ as follows. For any $b\in\baireinc$ let
$$F^X(b)(n)=\varphi_X{\restriction}[b(n),b(n+1))$$
for all $n\in\omega$. Put $$\Psi_-(X)=(k^X,\Phi\circ F^X),$$
where $\Phi$ is the mapping from \autoref{bairecoding}.

Now, we are going to define $\Psi_+$. Let $b\in\baireinc$ and $S\in\Swf(\omega,h)$. 
We define $\bar{S}_b\in\Swf(\omega,h)$ by
$$\bar{S}_b(n)=\largeset{t^{b,n}_K}{K\in S(n)\cap \kappa_{b,n} \text{ and } \forall j\in[b(n),b(n+1)):\frac{\abs{t^{b,n}_K(j)}}{2^{\abs{I_j}}}<\frac{\varepsilon_j}{(n+1)^3}},$$
where the function $\kappa_{b,n}$ and the enumeration $t_K^{b,n}$ are from \autoref{bairecoding}.

For every $n\in\omega$ and every $j\in [b(n),b(n+1))$ define
\begin{equation}\label{defpsiuni}
    \psi_{b,S}(j)=\bigcup_{t\in\bar{S}_b(n)} t(j).
\end{equation}

Note that $\psi_{b,S}\in\Sigma_{I,\varepsilon}$ since for $j\in[b(n),b(n+1))$ we have
$$\frac{\abs{\psi_{b,S}(j)}}{2^{\abs{I_j}}}< \abs{\bar{S}_b(n)}\cdot \frac{\varepsilon_j}{(n+1)^3}\leq \abs{S(n)}\cdot \frac{\varepsilon_j}{(n+1)^3}\leq (n+1)^2\cdot \frac{\varepsilon_j}{(n+1)^3}\leq \frac{\varepsilon_j}{n+1}.$$
Thus, we define $\Psi_+(b,S)=[\psi_{b,S}]_*$.

Let $X\in\Ewf_{I,\varepsilon}$ and $(b,S)\in \baireinc\times\Swf(\omega,h)$ be such that $\Psi_-(X)\sqsubset^\bullet (b,S)$, i.e., $k^X\leq^* b$ and $(\Phi\circ F^X)(b)\in^* S$. We shall show that $X\subseteq \Psi_+(b,S)$.

Since $k^X\leq^* b$ it follows that
\begin{equation}\label{domb}
   \forall^\infty n\in\omega\ \forall j\geq b(n): \frac{\abs{\varphi_X(j)}}{2^{\abs{I_j}}}<\frac{\varepsilon_j}{(n+1)^3}. 
\end{equation}
Denote by $\seq{K_n}{n\in\omega}$ the sequence $(\Phi\circ F^X)(b)$. By the fact that $\varphi_X\in\Sigma_{I,\varepsilon}$, it follows that
 $\varphi_X{\restriction}[b(n),b(n+1))\in\prod_{m\in[b(n),b(n+1))}\Pwf({^{I_m}2})$  for any $n\in\omega$. Therefore, we have
\begin{enumerate}[label=\normalfont(\roman*)]
    \item $\forall n\in\omega: K_n<\kappa_{b,n}$,
    \item $\forall n\in\omega: F^X(b)(n)=\varphi_X{\restriction}[b(n),b(n+1))=t^{b,n}_{K_n}$,
    \item\label{diffnot} $\forall n\in\omega\ \forall j\in[b(n),b(n+1)):\varphi_X(j)=t_{K_n}^{b,n}(j)$.
\end{enumerate}

Thus, it follows from $(\Phi\circ F^X)(b)=\seq{K_n}{n\in\omega}\in^* S$ that $\forall^\infty n\in\omega: K_n\in S(n)\cap\kappa_{b,n}$, by \eqref{domb} and \ref{diffnot} we have that $\forall^\infty n\in\omega: t^{b,n}_{K_n}=\varphi_X{\restriction}[b(n),b(n+1))\in\bar{S}_b(n)$. Consequently, $\forall^\infty j\in\omega :\varphi_X(j)\subseteq \psi_{b,S}(j)$ by \eqref{defpsiuni}. Applying \autoref{subsetpointwiseE} we get $X\subseteq [\varphi_X]_*\subseteq [\psi_{b,S}]_*=\Psi_+(b,S)$.

\ref{addcofb}: The proof is almost identical to part \ref{addcofa}, we just need to replace $\Ewf_{I,\varepsilon}$ with $\Swf_{I,\varepsilon}$, $[\cdot]_*$ with $[\cdot]_\infty$ and \autoref{subsetpointwiseE} with \autoref{subsetpointwiseS}.
\end{proof}

The following can be found in \cite{blass}.
\begin{fact}\label{productchar}
    Let $\Rbf=\langle X,Y,\sqsubset \rangle$ and $\Rbf'=\langle X',Y',\sqsubset' \rangle$ be relational systems. Then we have  $\dfrak(\Rbf;\Rbf')=\dfrak(\Rbf)\cdot\dfrak(\Rbf')$ and $\bfrak(\Rbf;\Rbf')=\min\{\bfrak(\Rbf),\bfrak(\Rbf')\}$.
\end{fact}

Recall that by \autoref{SEcomparison}, $\cof(\Swf_{I,\varepsilon})=\cof(\Ewf_{I,\varepsilon})$ and $\add(\Ewf_{I,\varepsilon})=\add(\Swf_{I,\varepsilon})$ for all $\Ewf$-contributive pairs $(I,\varepsilon)$. However, the next assertion holds for all $(I,\varepsilon)$. 

\begin{corollary}\label{a10}
For any $I\in\Ior$ and $\varepsilon\in\ell^1_+$ the following holds.
      \begin{enumerate}[label=\normalfont(\alph*)]
        \item\label{addEieNa} $\add(\Nwf)\leq \add(\Ewf_{I,\varepsilon}),\add(\Swf_{I,\varepsilon})$.
        \item\label{addEieNb} $\cof(\Ewf_{I,\varepsilon}),\cof(\Swf_{I,\varepsilon})\leq \cof(\Nwf)$.
    \end{enumerate}
\end{corollary}
\begin{proof}
We shall prove the assertion just for $\Ewf_{I,\varepsilon}$. The proof for $\Swf_{I,\varepsilon}$ is similar.  

If $(I,\varepsilon)$ is not $\Ewf$-contributive then $\add(\Nwf)\leq\infty=\add(\Ewf_{I,\varepsilon})$ and $\cof(\Ewf_{I,\varepsilon})=1\leq \cof(\Nwf)$. For $\Ewf$-contributive pairs we proceed as follows.

\ref{addEieNa}:  By~\autoref{hardtukey}~\ref{addcofa} and \autoref{productchar} we have $\min\{\bfrak,\blc_{\omega,h}\}\leq\add(\Ewf_{I,\varepsilon})$. Applying \autoref{addbloc} and the inequality $\add(\Nwf)\leq\bfrak$ we get $\min\{\bfrak,\blc_{\omega,h}\}=\min\{\bfrak,\add(\Nwf)\}=\add(\Nwf)$. 

\ref{addEieNb}: By~\autoref{hardtukey}~\ref{addcofa} and \autoref{productchar} we have $\cof(\Ewf_{I,\varepsilon})\leq \dfrak\cdot \dlc_{\omega,h}$. Applying \autoref{addbloc} and the fact that $\dfrak\leq \cof(\Nwf)$ we get $\dfrak\cdot \dlc_{\omega,h}=\dfrak\cdot\cof(\Nwf)=\cof(\Nwf)$. 
\end{proof}

\begin{corollary}\label{addS}
      For any $I\in\Ior$ and $\varepsilon\in\ell^1_+$,
      \vspace{0.1cm}
      \begin{enumerate}[label=\normalfont(\alph*)]
          \item $\add(\Nwf)\leq\min\set{\add(\Swf_{I,\varepsilon})}{I\in\Ior,\varepsilon\in\ell^1_+}$,
          \item $\cof(\Nwf)\geq\sup\set{\cof(\Swf_{I,\varepsilon})}{I\in\Ior,\varepsilon\in\ell^1_+}$.
      \end{enumerate}
\end{corollary}

One can ask whether the inequalities in \autoref{addS} could be changed to equalities. The answer is positive after assuming some additional inequalities between cardinal invariants. We start with two auxiliary results. The next lemma describes the null sets.

\begin{lemma}[{see e.g.~\cite[Lem.~2.5.1]{bartjud}}]\label{nullapprox}
    For any null set $G\subseteq {^\omega2}$ there is a sequence $\seq{F_n}{n\in\omega}\in \prod_{n\in\omega}\Pwf({^n2})$ with $\sum_{n\in\omega}\frac{|F_n|}{2^n}<\infty$ such that 
    \[
    G\subseteq \set{x\in {^\omega2}}{\exists^\infty n: x{\restriction} n\in F_n}.
    \]
\end{lemma}

In order to be brief, we shall use the following way of coding of small sets:

\begin{definition}[{\cite[Def.~6.5]{CM23}}]\label{def:2small}
We say that $\tbf=(L,\varepsilon,\varphi,\psi)$ is a~\emph{\twosmall} if:
\begin{enumerate}[label =\normalfont (T\arabic*)]
    \item\label{def:2smalla} $ L\in\Ior$, $\varepsilon\in\ell^1_+$, and we denote $L_{2k}=[n_k,m_k)$ and $L_{2k+1}=[m_k,n_{k+1})$ (so $n_k < m_k <n_{k+1}$), and define $I:=\la I_k:\, k<\omega\ra$ and $I':=\la I'_k:\, k<\omega\ra$ by $I_k:=[n_k,n_{k+1})$ and $I'_k:=[m_k,m_{k+1})$.
    \item\label{def:2smallb} $\varphi\in \Sigma_{I,\varepsilon}$, $\psi\in\Sigma_{I',\varepsilon}$.
\end{enumerate}
\end{definition}

The following is a~strengthening of a~similar result in \cite{bartosmall}. Our proof is a~modification of the corresponding result in \cite{bartosmall} and follows the proof of Lemma~6.6 in \cite{CM23}.

\begin{lemma}\label{lem:U2small}
  Let $\set{G_\alpha}{\alpha<\kappa}\subseteq\Nwf$, $\varepsilon\in\ell^1_+$ be decreasing and let $D\subseteq\Ior$ be dominating in $\Ior$. If $\kappa<\bfrak$, then there is some $L\in D$ and some \twosmall s $\tbf^\alpha=(L,\varepsilon,\varphi^\alpha,\psi^\alpha)$ for $\alpha<\kappa$,  such that
      \[G_\alpha\subseteq [\varphi^\alpha]_\infty\cup[\psi^\alpha]_\infty \text{ for every }\alpha<\kappa.\]
\end{lemma}
\begin{proof}
   By~\autoref{nullapprox} there are sequences $\seq{F^\alpha_n}{n\in\omega}$ such that $G_\alpha\subseteq \set{x\in{^\omega2}}{\exists^\infty n: x{\restriction} n\in F^\alpha_n}$ for each $\alpha<\kappa$.   Define sequences $n^\alpha_k, m^\alpha_k$ as follows: $n^\alpha_0=0$,
$$m^\alpha_k=\min\largeset{j>n^\alpha_k}{2^{n^\alpha_k}\cdot \sum_{i=j}^\infty\frac{|F^\alpha_i|}{2^i}<\frac{\varepsilon_k}{k}},$$
$$n^\alpha_{k+1}=\min\largeset{j>m^\alpha_k}{2^{m^\alpha_k}\cdot\sum_{i=j}^\infty\frac{|F^\alpha_i|}{2^i}<\frac{\varepsilon_k}{k}},$$
for each $\alpha<\kappa$, $k\in\omega$.

Now we let $J^\alpha=\set{[n^\alpha_{2k}, n^\alpha_{2(k+1)})}{k\in\omega}$ and since $\kappa<\bfrak$ we can find $L\in D$ such that $J^\alpha\sqsubseteq L$ for each $\alpha<\kappa$. Denote $L_{2k}=[n_k,m_k)$ and $L_{2k+1}=[m_k,n_{k+1})$. Now define
$I=\set{[n_k,n_{k+1})}{k\in\omega}$ and $I'=\set{[m_k,m_{k+1})}{k\in\omega}$. By the definition of $\sqsubseteq$-relation, we have that for any $\alpha<\kappa$ there is an $i_\alpha\in\omega$ such that 
$$\forall k>i_\alpha\ \exists j,j'\geq k: [n^\alpha_j,m^\alpha_j)\subseteq L_{2k}\text{ and }[n^\alpha_{j'}, m^\alpha_{j'+1})\subseteq L_{2k+1}.$$
Note that we may assume $j,j'\geq k$, since by the definition of $J^\alpha$, there is a $k_\alpha$ such that $\forall k>k_\alpha$ $[n^\alpha_{2l}, n^\alpha_{2l+2})\subseteq L_k$ for some $l$, hence, $[n^\alpha_{2l}, m^\alpha_{2l}),[n^\alpha_{2l+1}, m^\alpha_{2l+1})\subseteq L_k$. Then the smallest (w.r.t. indices) possible subinterval  for $L_{k+1}$ is $[n^\alpha_{2l+2}, m^\alpha_{2l+2})$, for $L_{k+2}$ it is $[n^\alpha_{2l+4}, m^\alpha_{2l+4})$ etc. Clearly, indices of subintervals grow at least two times faster than indices of $L_k$'s. So, eventually $j$'s must catch $k$'s up.

Now, for any $\alpha<\kappa$ define sequences $\varphi^\alpha,\psi^\alpha$ as follows: $\varphi^\alpha(k)=\psi^\alpha(k)=\emptyset$ for $k\leq i_\alpha$, otherwise 
\begin{align*}
\varphi^\alpha(k)&=\set{s\in {^{I_k}}2}{\exists i\in [m_k,n_{k+1})\ \exists t\in F^\alpha_i:s{\restriction} (\dom(t)\cap I_k)=t{\restriction}(\dom(t)\cap I_k)},\\
\psi^\alpha(k)&=\set{s\in {^{I'_k}}2}{\exists i\in [n_{k+1},m_{k+1})\ \exists t\in F^\alpha_i:s{\restriction} (\dom(t)\cap I'_k)=t{\restriction}(\dom(t)\cap I'_k)}.
\end{align*}

For $k\leq i_\alpha$ we have $\frac{\abs{\varphi^\alpha(k)}}{2^{\abs{I_k}}}=0$. For $k>i_\alpha$
$$ [n^\alpha_j,m^\alpha_j)\subseteq [n_k, m_k)=L_{2k} \text{ and }  [m^\alpha_{j'}, n^\alpha_{j'+1})\subseteq [m_k, n_{k+1})=L_{2k+1}$$
for some $j,j'\geq k$. That is, $n_k\leq n^\alpha_j<m^\alpha_j\leq m_k$ and $m_k\leq m^\alpha_{j'}< n^\alpha_{j'+1}\leq n_{k+1}$. Therefore, since for any $k$ we have $$\abs{\varphi^\alpha(k)}\leq \sum_{i=m_k}^{n_{k+1}}\abs{F^\alpha_i}\cdot 2^{\abs{I_k}-(i-n_k)}$$ (for any $i\in[m_k,n_{k+1})$ there are $2^{\abs{I_k}-(i-n_k)}$ extensions of $F^\alpha_i$ inside $[n_k,n_{k+1})$), one can easily see that for all but finitely many $k\in\omega$ we have
$$\frac{|\varphi^\alpha(k)|}{2^{|I_k|}}\leq2^{n_k}\cdot \sum_{i=m_k}^{n_{k+1}}\frac{|F^\alpha_i|}{2^i}\leq 2^{n^\alpha_j}\cdot \sum_{i=m^\alpha_j}^{n_{k+1}} \frac{\abs{F^\alpha_i}}{2^i} <\frac{\varepsilon_j}{j}\leq\frac{\varepsilon_k}{j}\leq\frac{\varepsilon_k}{k}.$$
Hence, $\frac{|\varphi^\alpha(k)|}{2^{|I_k|}\cdot \varepsilon_k}\leq\frac{1}{k}\to 0$. Therefore $\varphi^\alpha\in \Sigma_{I,\varepsilon}$. In a similar way we can prove $\psi^\alpha\in\Sigma_{I',\varepsilon}$. 

It remains to show $G_\alpha\subseteq [\varphi^\alpha]_\infty\cup[\psi^\alpha]_\infty$. Let $x\in G_\alpha$ and $X=\set{n\in\omega}{x{\restriction} n\in F^\alpha_n}$. By our assumption about $G_\alpha$ we have $|X|=\omega$. Then one of the sets $X\cap \bigcup_{k\in\omega}[m_k,n_{k+1})$ or $X\cap\bigcup_{k\in\omega}[n_{k+1},m_{k+1})$ is infinite. WLOG assume the first case, i.e., for infinitely many $n$'s we have $n\in [m_k,n_{k+1})$ for some $k$ and $x{\restriction} n\in F^\alpha_n$. For those $k$'s, by the definition of $\varphi^\alpha(k)$ there is $s\in \varphi^\alpha(k)$ such that $s=x{\restriction}[n_k,n_{k+1})=x{\restriction} I_k$. So, we have $x{\restriction} I_k\in \varphi^\alpha(k)$. 
\end{proof}

We show that we can prove the reversed inequalities to those in \autoref{addS} under additional assumptions.
\begin{theorem}\label{Thm:cofadd}\
\begin{enumerate}[label=\rm(\alph*)]

   \item\label{ThmM:ac:aa} Assume that $\add(\Nwf)<\bfrak$. Then $\add(\Nwf)=\min\set{\add(\Swf_{I,\varepsilon})}{I\in\Ior\text{\ and\ }\varepsilon\in\ell^1_+}$.

    \item\label{ThmM:ac:ab}  Assume that $\cof(\Nwf)>\dfrak$. Then $\cof(\Nwf)=\sup\set{\cof(\Swf_{I, \varepsilon})}{I\in\Ior\text{\ and\ }\varepsilon\in\ell^1_+}$.
\end{enumerate}    
\end{theorem}
\begin{proof}
\ref{ThmM:ac:aa}: The inequality ``$\leq$" is clear from~\autoref{addS}. So we shall prove $\add(\Nwf)\geq\min\set{\add(\Swf_{I,\varepsilon})}{I\in\Ior,\varepsilon\in\ell^1_+}$. To see this, 
let $\Awf=\set{A_\alpha}{\alpha<\kappa}\subseteq \Nwf$ be an witness for $\add(\Nwf)$, i.e., $\abs{\Awf}=\add(\Nwf)$ and $\bigcup_{\alpha<\kappa} A_\alpha\notin\Nwf$. By \autoref{lem:U2small} there are $\varepsilon\in\ell^1_+$, $I,I'\in\Ior$ and also $\varphi^\alpha\in\Sigma_{I,\varepsilon}$, $\psi^\alpha\in\Sigma_{I',\varepsilon}$, such that $A_\alpha\subseteq [\varphi^\alpha]_\infty\cup [\psi^\alpha]_\infty$ for each $\alpha<\kappa$. Clearly, one of the sets $\bigcup_{\alpha<\kappa} [\varphi^\alpha]_\infty$ or $\bigcup_{\alpha<\kappa} [\psi^\alpha]_\infty$ is not null, otherwise $\bigcup_{\alpha<\kappa} [\varphi^\alpha]_\infty\cup[\psi^\alpha]_\infty \supseteq\bigcup_{\alpha<\kappa} A_\alpha$ would be null, which is not true by the assumption. Recall that $\Swf_{I,\varepsilon}\subseteq \Nwf$ for any $I\in\Ior$, $\varepsilon\in\ell^1_+$. Consequently, $\bigcup_{\alpha<\kappa} [\varphi^\alpha]_\infty\notin\Swf_{I,\varepsilon}$ or $\bigcup_{\alpha<\kappa} [\psi^\alpha]_\infty\notin\Swf_{I',\varepsilon}$. WLOG assume the first case. Then $\set{[\varphi^\alpha]_\infty}{\alpha<\kappa}$ is a family of sets from $\Swf_{I,\varepsilon}$ such that $\bigcup\set{[\varphi^\alpha]_\infty}{\alpha<\kappa}\notin\Swf_{I,\varepsilon}$ and has cardinality $\leq \add(\Nwf)$. We get
$\add(\Swf_{I,\varepsilon})\leq \abs{\set{[\varphi^\alpha]_\infty}{\alpha<\kappa}}\leq\add(\Nwf).$
\par
\ref{ThmM:ac:ab}: Since $\sup\set{\cof(\Swf_{I,\varepsilon})}{I\in\Ior,\varepsilon\in\ell^1_+}\leq\cof(\Nwf)$ we need to prove only the other inequality. Pick a dominating family $D$ in $\Ior$ such that $\abs{L_k}\geq  k+1$, $k\in\omega$, for all $L\in D$ and such that $\abs{D}=\dfrak$. For each $L\in D$ define partitions $I_L$ and $I_L'$ as in \autoref{def:2small}~\ref{def:2smalla}. Note that $\abs{I_{L,k}}\geq 2(k+1)$ and $\abs{I_{L,k}'}\geq 2(k+1)$ for each $k$. For each $L\in D$ denote by $\Wwf_L$, $\Wwf_L'$ witnesses for $\cof(\Swf_{I_L,2^{-k-1}})$ and $\cof(\Swf_{I_L',2^{-k-1}})$, respectively. Notice that $(I_L,2^{-k-1})$ and $(I_L',2^{-k-1})$ are $\Ewf$-contributive by \autoref{charofrEIe}~\ref{charofrEIe1lim} because $\lim_{k\to\infty}2^{\abs{I_{L,k}}}\cdot 2^{-k-1}\geq   2^{2k+2}\cdot 2^{-k-1}=2^{k+1}\to\infty$. Similarly $\lim_{k\to\infty}2^{\abs{I'_{L,k}}}\cdot \varepsilon_k\to\infty$.  Define 
$$\Awf=\bigcup_{L\in D}\set{A_0\cup A_1}{A_0\in \Wwf_L, A_1\in\Wwf_L'}.$$

Let $\Bwf\subseteq \Nwf$ be a base of $\Nwf$. By \autoref{lem:U2small} for each $B\in\Bwf$ there is $L\in D$ such that $B=B_0\cup B_1$ where $B_0\in \Swf_{I_L,2^{-k-1}}$ and $B_1\in \Swf_{I_L',2^{-k-1}}$ and consequently (since $\Wwf_L$ and $\Wwf_L'$ are bases) there are $A_0^B\in\Wwf_L$ and $A_1^B\in\Wwf_L'$ such that  $B\subseteq A_0^B\cup A_1^B\in\Awf$. Thus, $\Awf$ is a base of $\Nwf$. Note that
$$\abs{\Awf}\leq \abs{D}\cdot\sup\set{\abs{\Wwf_L}\cdot\abs{\Wwf_L'}}{L\in D}\leq \dfrak\cdot \sup\set{\cof(\Swf_{J,\varepsilon})}{J\in\Ior,\varepsilon\in \ell^1_+}.$$
Since $\dfrak<\cof(\Nwf)\leq\abs{\Awf}$, we have that
$$\abs{\Awf}\leq \max\{\dfrak, \sup\set{\cof(\Swf_{J,\varepsilon})}{J\in\Ior,\varepsilon\in \ell^1_+}\}=\sup\set{\cof(\Swf_{J,\varepsilon})}{J\in\Ior,\varepsilon\in \ell^1_+}.$$
Consequently, $\cof(\Nwf)\leq\abs{\Awf}\leq\sup\set{\cof(\Swf_{J,\varepsilon})}{J\in\Ior,\varepsilon\in \ell^1_+}$. 
\end{proof}

\section{Covering and uniformity of the new ideals}

This section is mainly devoted to proving~\autoref{Thm:a1} and ~\autoref{Thm:m2}. We begin with the following observation concerning localization and anti-localization cardinals.

\begin{remark}
Localization and anti-localization cardinals defined in \autoref{defloc} depend on the parameter $b = \seq{b(n)}{n\in\omega}$. The following lemma shows that only the cardinalities of $b(n)$'s are important for investigating localization and anti-localization cardinals. Therefore, we  can formulate all our results assuming $b\in{^\omega\omega}$ (identifying each natural number $b(n)$ with $\{0,1,\dotso,b(n)-1\}$). However, in the proofs we usually work with a suitable sequence of sets instead of sequence of natural numbers. The lemma is due to \cite{CM}.    
\end{remark}

\begin{lemma}\label{cardinaltyDep}
Let $b,b'$ be sequences of length $\omega$ of non-empty sets and let $h,h'\in {^\omega\omega}$. If for all but finitely many $n\in\omega$, $\abs{b(n)}\leq \abs{b'(n)}$, and $h'\leq^* h$, then $\Lc(b,h)\leqT \Lc(b',h')$ and $\aLc(b',h')\leqT \aLc(b,h)$. In particular
\begin{enumerate}[label=\normalfont(\alph*)]
    \item $\blc_{b',h'}\leq\blc_{b,h}$ and $\dlc_{b,h}\leq\dlc_{b',h'}$,
    \item $\balc_{b,h}\leq\balc_{b',h'}$ and $\dalc_{b',h'}\leq\dalc_{b,h}$.
\end{enumerate}
\end{lemma}

The inequalities of~\autoref{Thm:a1} shall be inferred from the following result.

\begin{lemma}\label{lem:alcI}
\
\begin{enumerate}[label=\rm(\alph*)]
    \item\label{alcI-a}  Let $I\in\Ior$ and $\varepsilon\in\ell^1_+$. Then there are $h, b\in\baire$ such that $\sum_{n\in\omega}\frac{h(n)}{b(n)}<\infty$ and $\aLc(b,h)^{\perp}\leqT \Cbf_{\Swf_{I,\varepsilon}}.$ In particular, $\balc_{b,h}\leq \cov(\Swf_{I,\varepsilon})$ and $\non(\Swf_{I,\varepsilon})\leq\dalc_{b,h}$.

    \item\label{alcI-b} Let $h, b\in\baire$ be such that $\sum_{n\in\omega}\frac{h(n)}{b(n)}<\infty$. Then there are $I\in\Ior$ and $\varepsilon\in\ell^1_+$ such that $\cov(\Swf_{I,\varepsilon})\leq\balc_{b,h}$ and $\dalc_{b,h}\leq\non(\Swf_{I,\varepsilon})$.  
\end{enumerate}
\end{lemma}
\begin{proof}
\ref{alcI-a}: If $I,\varepsilon$ are not $\Swf^\star$-contributive then pick arbitrary $h,b$ such that the corresponding relational system $\aLc(b,h)$ is defined and the proof is done. Otherwise, for $n\in\omega$ define $b(n)={^{I_n}2}$ and $h(n)=\lfloor|b(n)|\cdot\varepsilon_n\rfloor$. Notice that $\sum_{n\in\omega}\frac{h(n)}{\abs{b(n)}}<\infty$ because $\varepsilon\in\ell^1_+$. So it remains to prove that $\aLc(b,h)^{\perp}\leqT \Cbf_{\Swf_{I,\varepsilon}}.$

Let $X\in\Swf_{I, \varepsilon}$. Choose $\varphi_X\in\Sigma_{I,\varepsilon}$ such that $X\subseteq[\varphi_X]_\infty$. Since $\lim_{n\to\infty}\frac{\abs{\varphi_X(n)}}{2^{\abs{I_n}}\cdot \varepsilon_n}=0$, we can find $n_0$ such that $\abs{\varphi_X(n)}<2^{\abs{I_n}}\cdot \varepsilon_n$ for all $n>n_0$. Consequently, since $\abs{\varphi(n)}\in\omega$ for all $n$, we have that $\forall n>n_0:\abs{\varphi_X(n)}\leq \lfloor 2^{\abs{I_n}}\cdot \varepsilon_n \rfloor$. Next, for $n\in\omega$ define $\varphi_X^*$ by 
\[\varphi_X^*(n)=\begin{cases}\emptyset,& n\leq n_0,\\
\varphi_X(n),& n > n_0.
\end{cases}\]
It is clear that $\varphi_X^*\in \Swf(b,h)$ and $X\subseteq [\varphi_X^*]_\infty$, so put $\Psi_+(X)=\varphi_X^*$. On the other hand, define $\Psi_-\colon \prod b\to {^\omega2}$ by $\Psi_-(s)=s_0^\frown s_1^\frown s_2^\frown\dotso$. It is clear that $\Psi_-(s)\in X$ implies $s\in^\infty \Psi_+(X)$, which guarantees that $(\Psi_-,\Psi_+)$ is the desired Tukey connection.
\par
\ref{alcI-b}: We will show that there are $(I,\varepsilon)\in\Ior\times\ell^1_+$ and a sequence $b'$ with $\abs{b'(n)}\leq b(n)$ for all $n\in\omega$, such that $\cov(\Swf_{I,\varepsilon})\leq \balc_{b',h}$ and $\dalc_{b',h}\leq\non(\Swf_{I,\varepsilon})$. Consequently, by \autoref{cardinaltyDep} 
\[\textrm{$\cov(\Swf_{I,\varepsilon})\leq\balc_{b',h}\leq \balc_{b,h}$ and 
$\dalc_{b,h}\leq \dalc_{b',h}\leq\non(\Swf_{I,\varepsilon})$}.
\]

    Let $\seq{I_n}{n\in\omega}$ be an interval partition of $\omega$ such that $|I_n|=\lfloor\log_2 b(n)\rfloor$. For $n\in\omega$ define $b'(n)={^{I_n}2}$. Note that $$\abs{b'(n)}=2^{\lfloor\log_2 b(n)\rfloor}\leq 2^{\log_2 b(n)}= b(n)<2^{\abs{I_n}+1}=2\cdot \abs{b'(n)}$$ and thus, also $\frac{b(n)}{2}<\abs{b'(n)}$ for all $n\in\omega$. Therefore, $\sum_{n\in\omega}\frac{h(n)}{\abs{b'(n)}}\leq\sum_{n\in\omega}\frac{2h(n)}{b(n)}<\infty$.

Now, we will find $\varepsilon$ such that $\sum_{n\in\omega}\varepsilon_n<\infty$ and $\lim_{n\to\infty}\frac{h(n)}{|b'(n)|\cdot\varepsilon_n}=0$. 
By \autoref{a9} there is $\delta\to\infty$ such that $\sum_{n\in\omega}\frac{h(n)}{|b'(n)|}\delta_n<\infty$. Put $\varepsilon_n:=\frac{h(n)}{|b'(n)|}\delta_n$ for every $n\in\omega$.

Now, we will show $ \Cbf_{\Swf_{I,\varepsilon}}\leqT\aLc(b',h)^{\perp}$, i.e., $\la{^\omega2},\Swf_{I,\varepsilon},\in\ra\leqT \la\prod b',\Scal(b',h),\in^\infty\ra$. We need to define $\Psi_-:\cantor\to\prod b'$ and $\Psi_+:\Swf(b',h)\to\Swf_{I,\varepsilon}$.

 For $x\in\cantor$, define $\Psi_-(x)(n):=x{\restriction} I_n$. Note that $\Scal(b',h)\subseteq \Sigma_{I,\varepsilon}$. Indeed, if $\varphi\in\Scal(b',h)$ then $\varphi(n)\subseteq {^{I_n}2}$ and $\frac{\abs{\varphi(n)}}{2^{\abs{I_n}}\cdot\varepsilon_n}=\frac{\abs{\varphi(n)}}{\abs{b'(n)}\cdot\varepsilon_n}\leq \frac{h(n)}{\abs{b'(n)}\cdot\varepsilon_n}=\frac{1}{\delta_n}\to 0$. Thus, we define $\Psi_+$ naturally by $\Psi_+(\varphi)=[\varphi]_\infty$ for any $\varphi\in\Scal(b',h)$.

Let $x\in{^\omega2}$ and $\varphi\in\Scal(b',h)$ be such that $\Psi_-(x)\in^\infty \varphi$. Then, clearly, $x\in \Psi_+(\varphi)=[\varphi]_\infty$.
\end{proof}

We are ready to show~\autoref{Thm:a1}:

\begin{theorem}\label{thm:alocsmall}
\
\begin{enumerate}[label=\rm(\alph*)]
\item\label{alcChA} $\min\set{\cov(\Swf_{I, \varepsilon})}{I\in\Ior\text{\ and\ }\varepsilon\in\ell^1_+}=\min\set{\balc_{b,h}}{b,h\in{^\omega\omega}\text{\ and\ } \sum_{n\in\omega}\frac{h(n)}{b(n)}<\infty}$.
\item\label{alcChB} $\sup\set{\non(\Swf_{I, \varepsilon})}{I\in\Ior\text{\ and\ }\varepsilon\in\ell^1_+}=\sup\set{\dalc_{b,h}}{b,h\in{^\omega\omega}\text{\ and\ } \sum_{n\in\omega}\frac{h(n)}{b(n)}<\infty}$.
\end{enumerate}    
\end{theorem}
\begin{proof}
By \autoref{lem:alcI}\ref{alcI-a} we have the inequalities ``$\geq$'' in \ref{alcChA} and ``$\leq$'' in \ref{alcChB}. On the other hand, by applying~\autoref{lem:alcI}\ref{alcI-b}, we have ``$\leq$'' in \ref{alcChA} and ``$\geq$'' in \ref{alcChB}.
\end{proof}

As a direct consequence of the theorem above, we provide another characterization of $\cov(\Nwf)$ and $\non(\Nwf)$ using the  $\sigma$-ideals $\Swf_{I,\varepsilon}$ as follows:

\begin{corollary}\label{cov_non}\
\begin{enumerate}[label=\rm(\alph*)]
    \item  Assume that $\cov(\Nwf)<\bfrak$. Then $\cov(\Nwf)=\min\set{\cov(\Swf_{I, \varepsilon})}{I\in\Ior\text{\ and\ }\varepsilon\in\ell^1_+}.$\smallskip
    
    \item Assume that $\non(\Nwf)>\dfrak$. Then $\non(\Nwf)=\sup\set{\non(\Swf_{I, \varepsilon})}{I\in\Ior\text{\ and\ }\varepsilon\in\ell^1_+}.$
\end{enumerate}    
\end{corollary}
\begin{proof}
By \cite{bartosmall}, we have the following, see also~\cite[Thm.~6.1]{CM23} for more details.
\begin{enumerate}[label=\normalfont(\alph*)]
    \item If $\cov(\Nwf)<\bfrak$ then $\cov(\Nwf)=\min\set{\balc_{b,h}}{b,h\in\baire\ \text{ and } \sum_{n\in\omega}\frac{h(n)}{b(n)}<\infty}.$
    \item If $\non(\Nwf)>\dfrak$ then $\non(\Nwf)=\sup\set{\dalc_{b,h}}{b,h\in\baire\ \text{ and } \sum_{n\in\omega}\frac{h(n)}{b(n)}<\infty}.$
\end{enumerate}
To conclude the proof, it is enough to apply \autoref{thm:alocsmall}.
\end{proof}

\autoref{lem:alcI} focuses on the family $\Swf_{I,\varepsilon}$. We obtain similar results for $\Ewf_{I,\varepsilon}$ as shown in~\autoref{c:lc}.

\begin{lemma}\label{c:lc}\ 
\begin{enumerate}[label=\rm(\alph*)]
    \item\label{c-1} Let $I\in\Ior$ and $\varepsilon\in\ell^1_+$. Then there are $h, b\in\baire$ such that $\sum_{n\in\omega}\frac{h(n)}{b(n)}<\infty$ and $\Lc(b,h)\leqT\Cbf_{\Ewf_{I,\varepsilon}}$ In particular, $\non(\Ewf_{I,\varepsilon})\leq\blc_{b,h}$ and $\dlc_{b,h}\leq \cov(\Ewf_{I,\varepsilon})$.   

    \item\label{c0}  Let $h, b\in\baire$ be such that $\sum_{n\in\omega}\frac{h(n)}{b(n)}<\infty$. Then there are $I\in\Ior$ and $\varepsilon\in\ell^1_+$ such that $\blc_{b,h}\leq\non(\Ewf_{I,\varepsilon})$ and $\cov(\Ewf_{I,\varepsilon})\leq\dlc_{b,h}$.
\end{enumerate}
\end{lemma}
\begin{proof}
Similar to the proof of \autoref{lem:alcI}.
\end{proof}

\autoref{Thm:m2} follows directly from the result mentioned above.

\begin{theorem}\label{a11}\
\begin{enumerate}[label=\normalfont(\alph*)]
    \item $\min\set{\cov(\Ewf_{I, \varepsilon})}{I\in\Ior\text{\ and\ }\varepsilon\in\ell^1_+}=\min\largeset{\dlc_{b,h}}{b,h\in{^\omega\omega}\text{\ and\ } \sum_{n\in\omega}\frac{h(n)}{b(n)}<\infty}.$
    \item $\sup\set{\non(\Ewf_{I, \varepsilon})}{I\in\Ior\text{\ and\ }\varepsilon\in\ell^1_+}=\sup\largeset{\blc_{b,h}}{b,h\in{^\omega\omega}\text{\ and\ } \sum_{n\in\omega}\frac{h(n)}{b(n)}<\infty}.$
\end{enumerate}
\end{theorem}
\begin{proof}
The proof is analogous to the proof of \autoref{thm:alocsmall}.
\end{proof}

The last provable result we are going to mention in this paper is establishing a connection between the uniformity and covering of $\Ewf_{I,\varepsilon}$ and the uniformity and covering of the $\sigma$-ideal of null-additive sets. The sum on $\cantor$ is the coordinate-wise sum modulo $2$ and the addition of sets is defined in a natural way.

\begin{definition}\label{def:na}
A set $X\subseteq\cantor$ is \emph{null-additive} if $X+A\in\Nwf$ for all $A\in \Nwf$. Denote by $\NAwf$ the $\sigma$-ideal of the null-additive sets in $\cantor$.
\end{definition}

We employ the following characterization of $\NAwf$ due to Shelah \cite{shelahnulladd}.

\begin{theorem}[{\cite{shelahnulladd}, see also~\cite[Thm.~2.7.18~(3)]{bartjud}}]\label{NAchar}
A set $X\subseteq {^\omega2}$ is null-additive if and only if for all $\seq{I_n}{n\in\omega}\in\Ior$ there is $\varphi\in \prod_{n\in\omega}\Pwf(^{I_n}2)$ such that 
\begin{enumerate}[label=\normalfont(\alph*)]
    \item\label{NAa} $\forall n\in\omega:\abs{\varphi(n)}\leq n$,
    \item\label{NAb} $\forall x\in X\ \forall^\infty n\in\omega: x{\restriction} I_n\in\varphi(n)$.
\end{enumerate}
\end{theorem}

Using the previous characterization of $\NAwf$, we easily get:

\begin{prop}\label{NAbound}\
\begin{enumerate}[label=\normalfont(\alph*)]
    \item  $\non(\NAwf)\leq \sup\set{\non(\Ewf_{I,\varepsilon})}{I\in\Ior\mathrm{\ and\ }\varepsilon\in\ell^1_+}$. 
    \item $\cov(\NAwf)\geq \min\set{\cov(\Ewf_{I,\varepsilon})}{I\in\Ior\mathrm{\ and\ } \varepsilon\in\ell^1_+}$.
\end{enumerate}

\end{prop}
\begin{proof}
We will show that there are $I\in\Ior, \varepsilon\in\ell^1_+$, such that $\non(\NAwf)\leq\non(\Ewf_{I,\varepsilon})$ and $\cov(\Ewf_{I,\varepsilon})\leq \cov(\NAwf)$. Let $I\in\Ior$ be such that $\abs{I_0}$ is an arbitrary natural number and $\abs{I_n}\geq \log_2(n^2\cdot 2^n)$,  (i.e., $2^{\abs{I_n}}\geq n^2\cdot 2^n$) for all $n\in\omega\smallsetminus\{0\}$. Define $\varepsilon\in\ell^1_+$ by $\varepsilon_n=2^{-n}$ for all $n\in\omega$. Note that such a pair is $\Ewf$-contributive by \autoref{charofrEIe}~\ref{charofrEIe1lim}, since 
$$\frac{1}{2^{\abs{I_n}}\cdot \varepsilon_n}\leq \frac{1}{n^2\cdot 2^n\cdot 2^{-n}}=\frac{1}{n^2}\to 0.$$

We shall show $\Cbf_{\Ewf_{I,\varepsilon}}\leqT\Cbf_{\NAwf}$. Define $\Psi_-\colon{^\omega 2}\to {^\omega 2}$ by putting $\Psi_-=\id_{{^\omega 2}}$. The $\Psi_+\colon\NAwf\to\Ewf_{I,\varepsilon}$ is defined as follows. For any $X\in\NAwf$, there is $\varphi_X\in\prod_{n\in\omega}\left({^{I_n}}2\right)$ satisfying conditions $\ref{NAa}$ and \ref{NAb} of \autoref{NAchar}. Note that $\varphi_X\in\Sigma_{I,\varepsilon}$ by
$$\frac{\abs{\varphi_X(n)}}{2^{\abs{I_n}}\cdot\varepsilon_n}\leq\frac{n}{2^{\abs{I_n}}\cdot\varepsilon_n}\leq \frac{n}{n^2\cdot2^{n}\cdot 2^{-n}}=\frac{1}{n}\to 0.$$ Hence, put $\Psi_+(X)=[\varphi_X]_*$.

Let $y\in{^\omega2}$, $X\in \NAwf$ be such that $\Psi_-(y)=y\in X$. Then, of course, $y\in [\varphi_X]_*=\Psi_+(X)$ because $X\subseteq [\varphi_X]_*$.
\end{proof}

\section{Consistency results}\label{sec:cs}

In this section, different forcing models are used to show that ZFC cannot prove more inequalities of the cardinal characteristics associated with our new ideals with the cardinals in Cicho\'n's diagram (see~\autoref{cichonext}). As usual, we start looking at the behavior of the cardinal characteristics associated with our ideals in well-known forcing models.

The following items contain details that we skip most of, however, details can be found in the references. We consider just contributive pairs $(I,\varepsilon)$ (in fact, $\Ewf$-contributive) throughout the whole section.

\begin{enumerate}[label=\rm(M\arabic*)]
    \item\label{cohen} Let $\lambda$ be an infinite cardinal such that $\lambda^{\aleph_0}=\lambda$ and let $h$ and $b$ be in $\baire$ satisfying the assumptions of~\autoref{lem:alcI}~\ref{alcI-b}. Then $\Cor_\lambda$ forces $\non(\Mwf)=\balc_{b,h}=\aleph_1<\cov(\Mwf)=\dalc_{b,h}=\cfrak=\lambda$ (see e.g.~\cite{Brendlecurse}). In particular, $\Cor_\lambda$ forces $\add(\Ewf_{I,\varepsilon})=\non(\Ewf_{I,\varepsilon})=\aleph_1$ and $\cov(\Ewf_{I,\varepsilon})=\cof(\Ewf_{I,\varepsilon})=\lambda$ for any $I$ and $\varepsilon$. On the other hand, by using ~\autoref{lem:alcI}~\ref{alcI-b}, we can find $I_{b,h}\in\Ior$ and $\varepsilon_{b,h}\in\ell^1_+$ such that $\dalc_{b,h}\leq\non(\Swf_{I_{b,h},\varepsilon_{b,h}})$ and $\cov(\Swf_{I_{b,h},\varepsilon_{b,h}})\leq\balc_{b,h}$. Hence, in $V^{\Cor_\lambda}$, $\non(\Swf_{I_{b,h},\varepsilon_{b,h}})=\lambda$ and $\cov(\Swf_{I_{b,h},\varepsilon_{b,h}})=\aleph_1$.

    \item\label{m2} Let $\aleph_1\leq\kappa\leq\lambda=\lambda^{\aleph_0}$ with $\kappa$ regular and let $\varrho,\rho\in\omega^\omega$ be such that $\sum_{i<\omega}\frac{\rho(i)^i}{\varrho(i)}<\infty$. Then by using~\autoref{lem:alcI}~\ref{alcI-b}, there are $I_{\varrho,\rho^\id}\in\Ior$ and $\varepsilon_{\varrho,\rho^\id}\in\ell^1_+$ such that $\cov(\Swf_{I_{\varrho,\rho^\id},\varepsilon_{\varrho,\rho^\id}})\leq\balc_{\varrho,\rho^\id}$ and $\dalc_{\varrho,\rho^\id}\leq\non(\Swf_{I_{\varrho,\rho^\id},\varepsilon_{\varrho,\rho^\id}})$.\footnote{The function $\rho^\id$ is defined by $\rho^\id(i)=\rho(i)^i$ for all $i\in\omega$.}  Now let $\Eor_{\pi}$ be a FS iteration of eventually different real forcing of length $\pi=\lambda\kappa$. Then, in $V^{\Eor_{\pi}}$, $\cov(\Nwf)=\bfrak=\aleph_1\leq\non(\Ewf)=\cov(\Ewf)=\kappa\leq\dfrak=\lambda$ (see e.g.~\cite{CaraboutE}). Moreover, by~\cite[Lem.~4.24]{CM}, $\Eor_{\pi}$ forces that $\balc_{\varrho,\rho^\id}=\aleph_1$ and $\dalc_{\varrho,\rho^\id}\leq\lambda$. Hence, in $V^{\Eor_{\pi}}$, $\cov(\Swf_{I_{\varrho,\rho^\id},\varepsilon_{\varrho,\rho^\id}})=\aleph_1$ and $\non(\Swf_{I_{\varrho,\rho^\id},\varepsilon_{\varrho,\rho^\id}})=\lambda$. On the other hand, given $I\in\Ior$ and $\varepsilon\in\ell^1_+$, by~\autoref{c:lc}~\ref{c-1}, we can find $b_{I,\varepsilon}$ and $h_{I,\varepsilon}$ such that $\non(\Ewf_{I,\varepsilon})\leq\blc_{b_{I,\varepsilon},h_{I,\varepsilon}}$ and $\dlc_{b_{I,\varepsilon},h_{I,\varepsilon}}\leq \cov(\Ewf_{I,\varepsilon})$. Then, in $V^{\Eor_{\pi}}$, we have that $\non(\Ewf_{I,\varepsilon})=\aleph_1$ and $\cov(\Ewf_{I,\varepsilon})=\cfrak=\lambda$ because $\Eor_{\pi}$ forces $\blc_{b,h}=\aleph_1$ and $\dlc_{b,h}=\lambda$ for all $b,h\in\baire$ with $0<h<b$ since $\Eor$ is $\sigma$-centered (see~\cite[Lem.~4.24]{CM}).

    \item Assume $\aleph_1\leq\kappa\leq\lambda=\lambda^{\aleph_0}$ with $\kappa$ regular. Let $\Bor_{\pi}$ be a FS iteration of random forcing of length $\pi=\lambda\kappa$. Then, in $V^{\Bor_{\pi}}$, $\non(\Ewf)=\ \bfrak=\aleph_1\leq\cov(\Nwf)=\non(\Nwf)=\kappa\leq\cov(\Ewf)=\ \dfrak=\cfrak=\lambda$ (see \cite{BS1992} or see also~\cite[Thm.~5.6]{Car4E}).
    
    In particular, in $V^{\Bor_{\pi}}$, we have that  $\add(\Swf_{I,\varepsilon})=\non(\Ewf_{I,\varepsilon})=\aleph_1$, 
     and $\cov(\Ewf_{I,\varepsilon})=\cof(\Swf_{I,\varepsilon})=\cfrak=\lambda$ for any $I$ and $\varepsilon$.

    \item Assume $\aleph_1\leq\kappa\leq\lambda=\lambda^{\aleph_0}$ with $\kappa$ regular. Let $\Dor_{\pi}$ be a FS iteration of Hechler forcing of length $\pi=\lambda\kappa$. Then, in $V^{\Dor_{\pi}}$, $\cov(\Nwf)=\aleph_1\leq\dfrak=\bfrak=\kappa\leq\non(\Nwf)=\cfrak=\lambda$ (see e.g.~\cite[Thm.~5]{mejiamatrix}). On the other hand, given $I\in\Ior$ and $\varepsilon\in\ell^1_+$, by~\autoref{c:lc}~\ref{c-1}, we can find $b_{I,\varepsilon}$ and $h_{I,\varepsilon}$ such that $\non(\Ewf_{I,\varepsilon})\leq\blc_{b_{I,\varepsilon},h_{I,\varepsilon}}$ and $\dlc_{b_{I,\varepsilon},h_{I,\varepsilon}}\leq \cov(\Ewf_{I,\varepsilon})$. Then, in $V^{\Dor_{\pi}}$, we have that $\add(\Swf_{I,\varepsilon})=\non(\Ewf_{I,\varepsilon})=\aleph_1$, and $\cof(\Swf_{I,\varepsilon})=\cov(\Ewf_{I,\varepsilon})=\lambda=\cfrak$ because  $\Dor_{\pi}$ forces $\blc_{b_{I,\varepsilon},h_{I,\varepsilon}}=\aleph_1$ and $\dlc_{b_{I,\varepsilon},h_{I,\varepsilon}}=\lambda$ by~\cite[Lem.~4.24]{CM}.  

    As in~\ref{m2}, we can find $I'$ and $\varepsilon'$ such that $\cov(\Swf_{I',\varepsilon'})=\aleph_1$ and $\non(\Swf_{I',\varepsilon'})=\lambda$ in  $V^{\Dor_{\pi}}$.

    \item\label{m5} Assume $\aleph_1\leq\kappa\leq\lambda=\lambda^{\aleph_0}$ with $\kappa$ regular. Let $h$ and $b$ be in $\baire$ satisfying the assumptions of~\autoref{c:lc}~\ref{c0} and let $\Por$ be a FS iteration of the $(b,h)$-localization forcing due to Brendle and Mej\'ia~\cite{BrM}. Then, in $V^\Por$, $\bfrak=\cov(\Nwf)=\aleph_1\leq\blc_{b,h}=\non(\Ewf)=\cov(\Ewf)=\dlc_{b,h}=\kappa\leq\dfrak=\cfrak=\lambda$ (see~\cite{Car4E} or see also~\cite{CaraboutE}). By~\autoref{c:lc}~\ref{c0}, we can find $I_{b,h}\in\Ior$ and $\varepsilon_{b,h}\in\ell^1_+$ such that $\blc_{b,h}\leq\non(\Ewf_{I_{b,h},\varepsilon_{b,h}})$ and $\cov(\Ewf_{I_{b,h},\varepsilon_{b,h}})\leq\dlc_{b,h}$. Hence, in the generic extension,  $\non(\Ewf_{I_{b,h},\varepsilon_{b,h}})=\cov(\Ewf_{I_{b,h},\varepsilon_{b,h}})=\kappa$.

    \item The model obtained by forcing with $\Aor_\kappa$ (random algebra for adding $\kappa$ random reals) with $\cf(\kappa)>\aleph_0$
over a model of CH. It satisfies that $\non(\Nwf)=\dfrak=\aleph_1$ and $\cov(\Nwf)=\cfrak=\kappa$ (see e.g.~\cite[Model 7.6.8]{bartjud}). In this model, $\non(\Swf_{I,\varepsilon})=\non(\Ewf_{I,\varepsilon})=\aleph_1$ and $\cov(\Swf_{I,\varepsilon})=\cov(\Ewf_{I,\varepsilon})=\kappa$ for any $I$ and $\varepsilon$.

 \item The dual random model. This model is obtained by forcing with $\Bor_{\omega_1}$ over a model for $\mathrm{MA}+\cfrak\geq\aleph_2$. It satisfies that $\non(\Nwf)=\aleph_1$ and $\cov(\Nwf)=\bfrak=\cfrak=\aleph_2$ (see e.g.~\cite[Model 7.6.7]{bartjud}). In this model, $\non(\Swf_{I,\varepsilon})=\non(\Ewf_{I,\varepsilon})=\aleph_1$ and $\cov(\Swf_{I,\varepsilon})=\cov(\Ewf_{I,\varepsilon})=\aleph_2$ for any $I$ and $\varepsilon$.

 \item\label{ma} Let $h, b\in\baire$ be such that $\sum_{n\in\omega}\frac{h(n)}{b(n)}<\infty$. Let $\Mor_{\omega_2}$ be a CS (countable support) iteration of length $\omega_2$ of Mathias forcing. Then, in $V^{\Mor_{\omega_2}}$, $\add(\Ewf)=\cov(\Ewf)=\aleph_1$, and  $\non(\Ewf)=\cof(\Ewf)=\cfrak=\aleph_2$ (see e.g~\cite[Lem.~2.4]{CaraboutE}). It is also well-known that $\Mor_{\omega_2}$ forces 
$\bfrak=\aleph_2$ (see e.g.~\cite[Sec. 7.4.A]{bartjud}. 

Moreover,  since $\Mor$ satisfies the Laver property (see~\cite[Sec.~7.4.A]{bartJudah}), we get $\dlc_{b,h}=\aleph_1$ in $V^{\Mor_{\omega_2}}$ (see~\cite[Theorem~4.3]{kada}), consequently, in $V^{\Mor_{\omega_2}}$, by applying~\autoref{c:lc}~\ref{c0}, we can find $I_{b,h}\in\Ior$ and $\varepsilon_{b,h}\in\ell^1_+$ such that $\cov(\Ewf_{I_{b,h},\varepsilon_{b,h}})\leq\dlc_{b,h}$. Hence, $\cov(\Ewf_{I_{b,h},\varepsilon_{b,h}})=\aleph_1$.

 \item\label{miller} Let $h, b\in\baire$ be such that $\sum_{n\in\omega}\frac{h(n)}{b(n)}<\infty$. Let $\MIor_{\omega_2}$ be a CS iteration of length $\omega_2$ of Miller forcing. Then, in $V^{\MIor_{\omega_2}}$, $\add(\Ewf)=\cov(\Ewf)=\non(\Ewf)=\aleph_1$, and  $\cof(\Ewf)=\aleph_2$ (see e.g~\cite[Lem.~2.5]{CaraboutE}). In particular, in $V^{\MIor_{\omega_2}}$, $\non(\Ewf_{I, \varepsilon})=\aleph_1$ 
 for any $I$ and $\varepsilon$. By arguing as in~\ref{ma}, we can find $I_{b,h}$ and $\varepsilon_{b,h}$ such that $\cov(\Swf_{I_{b,h},\varepsilon_{b,h}})=\cov(\Ewf_{I_{b,h},\varepsilon_{b,h}})=\aleph_1$ because Miller forcing satisfies the Laver property.

\end{enumerate}

In what follows, we aim to enhance~\ref{m5} by establishing the consistency of $\bfrak<\non(\Ewf_{I, \varepsilon})$ for any given $I$ and $\varepsilon$. To achieve this objective, we will introduce a specific forcing notion denoted as $\Por_{I,\varepsilon}$, which serves to increase the value of $\non(\Ewf_{I,\varepsilon})$ without adding dominating reals.

\begin{definition}\label{forP}
Let $I$ and $\varepsilon$ be as in~\autoref{a1}. Let $\Por_{I,\varepsilon}$ be a poset whose conditions are triples of the form  $(s,N,F)$ such that $s$ is a~finite sequence with $s(n)\subseteq{^{I_n}2}$ for $n\in|s|$, $N<\omega$, and  $F\in[\cantor]^{<\omega}$ satisfies $\forall n\geq |s|\colon N\cdot |F|\leq2^{|I_n|}\cdot\varepsilon_n$. We order $\Por_{I,\varepsilon}$ by $(t,M,H)\leq (s,N,F)$ if $s\subseteq t$, $M\geq N$, and  $F\subseteq H$, and the following requirements are met:
\begin{itemize}
    \item $\forall x\in F\,\forall n \in |t|\smallsetminus |s|\colon x{\upharpoonright} I_n\in t(n)$, 

    \item $\forall n\in |t|\smallsetminus |s|\,\colon N\cdot |t(n)|\leq2^{|I_n|}\cdot\varepsilon_n$.
\end{itemize}

Let $G\subseteq\Por$ be a  $\Por_{I,\varepsilon}$-generic over $V$, in $V[G]$, define 
\[\varphi_\gen=\bigcup\set{s}{\exists N, F\colon (s,N,F)\in G}.\]
Then $\varphi_\gen\in\Sigma_{I,\varepsilon}$ and, for every $x\in\cantor\cap V$, and for all but finitely many $n\in\omega\colon x{\upharpoonright}I_n\in \varphi_\gen(n)$. 
\end{definition}

\begin{lemma}
Let $I$ and $\varepsilon$ be as in~\autoref{a1}. Then $\Por_{I,\varepsilon}$ is $\sigma$-linked, so it is ccc.
\end{lemma}
\begin{proof}
Let $s\in\Seq_{<\omega}({^{I_n}2})$ and $N<\omega$, the set \[P_{I,\varepsilon}(s,N)=\set{(s',N',F)\in\Por_{I,\varepsilon}}{s=s'\,\&\,N=N'}\]  
is linked and $\bigcup_{s\in\Seq_{<\omega}({^{I_n}2}),\,N<\omega}P_{I,\varepsilon}(s,N)$  is dense in $\Por_{I,\varepsilon}$.
\end{proof}

In order to prove  $\bfrak<\non(\Ewf_{I, \varepsilon})$, we focus our attention on the notion of \emph{ultrafilter-limits} introduced in \cite{GMS}, and to the notion of \emph{filter-linkedness} by Mej\'ia~\cite{mejiavert}, with improvements in~\cite[Section 3]{BCM}.

 Given a poset $\Por$, the $\Por$-name $\dot{G}$ usually denotes the canonical name of the $\Por$-generic set. If $\bar{p}=\seq {p_n}{n<\omega}$ is a sequence in $\Por$, denote by $\dot{W}_\Por(\bar{p})$ the $\Por$-name of $\set{n<\omega}{p_n\in\dot{G}}$. When the forcing is understood from the context, we write just $\dot{W}(\bar{p})$.

\begin{definition}[{\cite{GMS,CMR2}}]\label{Def:GMS}
Let $\Por$ be a poset, let $D\subseteq \pts(\omega)$ be a non-principal ultrafilter, and let $\mu$ be an infinite cardinal.
\begin{enumerate}[label=\rm(\arabic*)]
\item\label{GMS1} A set $Q\subseteq \Por$ has \emph{$D$-limits} if there is a function $\lim^{D}\colon {}^{\omega}Q\to \Por$ and a $\Por$-name $\dot D'$ of an ultrafilter extending $D$ such that, for any $\bar q = \seq{q_i}{i<\omega}\in {}^{\omega}Q$,
\[{\lim}^{D}\, \bar q \Vdash \dot{W}(\bar{q})\in \dot D'.\]

\item A set $Q\subseteq \Por$ has \emph{uf-limits} if it has $D$-limits for any ultrafilter $D$.

\item The poset $\Por$ is \emph{$\mu$-$D$-$\lim$-linked} if $\Por = \bigcup_{\alpha<\mu}Q_\alpha$ where each $Q_\alpha$ has $D$-limits. We say that $\Por$ is \emph{uniformly $\mu$-$D$-$\lim$-linked} if, additionally, the $\Por$-name $\dot D'$ from~\ref{GMS1} only depends on $D$ (and not on $Q_\alpha$, although we have different limits for each $Q_\alpha$).

\item The poset $\Por$ is \emph{$\mu$-uf-$\lim$-linked} if $\Por = \bigcup_{\alpha<\mu}Q_\alpha$ where each $Q_\alpha$ has uf-limits. We say that $\Por$ is \emph{uniformly $\mu$-uf-$\lim$-linked} if, additionally, for any ultrafilter $D$ on $\omega$, the $\Por$-name $\dot D'$ from~\ref{GMS1}  depends only on $D$.
\end{enumerate}    
\end{definition}

For not adding dominating reals, the following notion is weaker than the previous one.

\begin{definition}[{\cite{mejiavert}}]\label{Def:Fr}
    Let $\Por$ be a poset and let $F$ be a filter on $\omega$. A set $Q\subseteq \Por$ is \emph{$F$-linked} if, for any $\bar p=\seq{p_n}{n<\omega} \in {}^{\omega}Q$, there is some $q\in \Por$ forcing that $F\cup \{\dot{W}(\bar p)\}$ generates a filter on $\omega$.
    We say that $Q$ is \emph{uf-linked (ultrafilter-linked)} if it is $F$-linked for any filter $F$ on $\omega$ containing the \emph{Fréchet filter} $\Fr:=\set{\omega\menos a}{a\in[\omega]^{<\aleph_0}}$.
    
    For an infinite cardinal $\mu$, $\Por$ is \emph{$\mu$-$F$-linked} if $\Por = \bigcup_{\alpha<\mu}Q_\alpha$ for some $F$-linked $Q_\alpha$ ($\alpha<\mu$). When these $Q_\alpha$ are uf-linked, we say that $\Por$ is \emph{$\mu$-uf-linked}.
\end{definition}

\begin{example}\label{ex:Fr}\
    \begin{enumerate}[label=\rm(\arabic*)]
        \item Random forcing is $\sigma$-uf-linked~\cite{mejiavert}, but it may not be $\sigma$-uf-$\lim$-linked (compare~\cite[Rem.~3.10]{BCM}). 

        \item\label{ex:Fr:b} Any poset of size ${\leq}\kappa$ is $\kappa$-Fr-linked (witnessed by its singletons). In particular, Cohen forcing is $\sigma$-Fr-linked.
    \end{enumerate}
\end{example}

It is clear that any uf-$\lim$-linked set $Q\subseteq \Por$ is uf-linked, which implies that $Q$ is $\Fr$-linked.

\begin{theorem}[{\cite[Thm.~3.30]{mejiavert}}]\label{mej:uf}
Let $\kappa\geq\aleph_1$ be a regular cardinal. If $\Por$ is $\kappa$-$\Fr$-linked poset, then $\Por$ forces $\bfrak\leq\kappa$.
\end{theorem}
\begin{proof}
    Although the conclusion of~\cite[Thm.~3.30]{mejiavert} is different, the same proof works.
\end{proof}

\begin{example}\label{exm:ufl}
The following are the instances of $\mu$-uf-lim-linked posets we use in our applications.
\begin{enumerate}[label=\rm (\arabic*)]

    \item Any poset of size $\mu$ is uniformly $\mu$-uf-lim-linked (because singletons are uf-lim-linked). In particular, Cohen
forcing is uniformly $\sigma$-uf-lim-linked.
    
    \item The standard eventually different real forcing notion is uniformly $\sigma$-uf-lim-linked (see~\cite{GMS,BCM}).
\end{enumerate}    
\end{example}

We aim to show that $\Por_{I,\varepsilon}$ is uniformly $\sigma$-uf-$\lim$-linked, witnessed by
\[P_{s, N, m}=\set{(t,M, F)\in \Por_{I,\varepsilon}}{\textrm{$s=t$, $N=M$, and $|F|\leq m$}}\]
for $s\in\Seq_{<\omega}({^{I_n}2})$ and $N<\omega$. Let $D$ be an ultrafilter on $\omega$, and $\bar{p}=\seq{p_n}{n\in\omega}$ be a 
sequence in $P_{s, N,m}$ with $p_n = (s, N, F_n)$. Consider the lexicographic order $\lhd$ on $\cantor$, and let $\set{x_{n,k}}{k<m_n}$ be a $\lhd$-increasing enumeration of $F_n$ where $m_n \leq m$. Next, find a unique  $m_*\leq m$ such that $A:=\set{n\in\omega}{m_n=m_*}\in D$. For each $k<m_*$, define $x_k:=\lim_n^D x_{n,k}$ in $\cantor$ where $x_k(i)$ is the unique member of $\{0,1\}$ such that $\set{n\in A}{x_{n,k}(i) = x_k(i)} \in D$ (this coincides with the topological $D$-limit). Therefore, we can think of $F:=\set{x_k}{k<m_*}$ as the $D$-limit of $\seq{F_n}{n<\omega}$, so we define $\lim^D \bar p:=(s,N,F)$. Note that $\lim^D \bar p \in P_{s, N, m}$. 

\begin{lemma}\label{uf:P_I,e}
The poset $\Por_{I,\varepsilon}$ is uniformly $\sigma$-uf-$\lim$-linked:
For any ultrafilter $D$ on $\omega$, there is a $\Por_{I,\varepsilon}$-name of an ultrafilter $\dot D'$ on $\omega$ extending $D$ such that,
for any $s\in \Seq_{<\omega}(\Pwf({^{I_n}2}))$, $N<\omega$, $m<\omega$ and $\bar p\in P_{s,N,m}$, $\lim^D \bar p \Vdash \dot W(\bar p)\in \dot D'$. 
\end{lemma}

To prove the theorem, it suffices to show the following:

\begin{claim}
      Assume $M<\omega$, $\set{(s_k,N_k,m_k)}{k<M}\subseteq\Seq_{<\omega}({^{I_n}2})\times\omega\times \omega$, $\seq{\bar{p}^k}{k<M}$ is such that each $\bar{p}^k=\seq{ p_{k,n}}{n<\omega}$ is a sequence in $P_{s_k,N_k,m_k}$, $q_k$ is the $D$-limit of $\bar{p}^k$ for each $k<M$, and $q\in\Por_{I,\varepsilon}$ is stronger than every $q_k$. Then, for any $a\in D$, there are $n<\omega$ and $q'\leq q$ stronger than $p_{k,n}$ for all $k<M$ (i.e.\ $q'$ forces $a\cap\bigcap_{k<M}\dot{W}(\bar{p}^k)\neq\emptyset$).
\end{claim}
\begin{proof}
Write the forcing conditions as $p_{k,n}=(s_k,N_k,F_{k,n})$, $q_k=(s_k,N_k,F_{k})$, where each $F_k = \set{x^k_j}{j<m_{*,k}}$ is the $D$-limit of $F_{k,n} = \set{x^{k,n}_j}{j<m_{*,k}}$ (increasing $\lhd$-enumeration) with $m_{*,k}\leq m_k$. Assume that $q = (s,N,F)\leq q_k$ in $\Por_{I,\varepsilon}$ for all $k<M$. Next, by strengthening $q$ if necessary, we assume that $\forall n\geq |s|\colon |F|+\sum_{k<M}m_{*,k}\leq \frac{2^{|I_n|}\cdot\varepsilon_n}{N}$. Then 
\[\textrm{$\forall j<m_{*,k}\,\forall x\in F\,\forall n \in |s|\smallsetminus |s_k|\colon x_j^k{\upharpoonright} I_n\in s(n)$,}\]
so
\[b_k := \big\{n<\omega \st\,  \textrm{$\forall j<m_{*,k}\,\forall x\in F\,\forall n \in |s|\smallsetminus |s_k|\colon x_j^{k,n}{\upharpoonright} I_n\in s(n)$} \big\}\in D.\]
Hence, $a\cap\bigcap_{k<M}b_k\neq\emptyset$, so choose $n\in a\cap\bigcap_{k<M}b_k$ and put $q'=(s,N,F')$ where $F':=F\cup \bigcup_{k<M}F_{k,n}$. This is a condition in $\Por_{I,\varepsilon}$ because $\forall n\geq |s|\colon|F'|\leq |F|+\sum_{k<M}m_{*,k}\leq\frac{2^{|I_n|}\cdot\varepsilon_n}{N}$. Furthermore, $q'$ is stronger than $q$ and $p_{n,k}$ for any $k<M$.
\end{proof}

\begin{theorem}
Let $I$ and $\varepsilon$ be as in~\autoref{a1}. Given regular uncountable $\kappa<\lambda=\lambda^{<\lambda}$, there is a ccc poset $\Por$ forcing $\non(\Ewf_{I,\varepsilon})=\lambda=\cfrak$ and $\bfrak=\kappa$.    
\end{theorem}
\begin{proof}
  Let $\Por=\seq{\Por_\alpha,\Qnm}{\alpha<\lambda}$  be a FS iteration of ccc forcing such that
  \begin{itemize}
      \item for $\alpha$ even, $\Vdash_\alpha\Qnm=\Por_{I,\varepsilon}$, the forcing defined in~\autoref{forP},
      \item for $\alpha$ odd, $\Vdash_\alpha\Qnm$ is a subposet of hechler forcing of size ${<}\kappa$.
  \end{itemize}
  First, notice that $\Por$ forces $\lambda\geq\cfrak$. Guarantee that we go through all such small subposets of Hechler forcing by a book-keeping argument. Then it is not hard to see that $\Por$ forces $\kappa\leq\bfrak$. On the other hand, since the iterands of the FS iteration that determine $\Por$ are $\kappa$-Fr-linked (when $\alpha$ is even, it follows by~\autoref{uf:P_I,e} and, when $\alpha$ is odd, it follows by~\autoref{ex:Fr}~\ref{ex:Fr:b}), by applying~\autoref{mej:uf} we obtain that $\Por$ forces $\bfrak\leq\kappa$.

As we iteratively add slaloms $\varphi_\gen\in\Sigma_{I,\varepsilon}$ which for all but finitely many $n\in\omega$, $x{\upharpoonright}I_n\in \varphi_\gen(n)$ for any $x\in \baire$ in the ground model, then  $\Por$ forces $\non(\Ewf_{I,\varepsilon})\geq\lambda$, so along with $\Vdash_\Por\lambda\geq\cfrak$, we obtain that $\Por$ forces $\non(\Ewf_{I,\varepsilon})\geq\lambda\geq\cfrak\geq\non(\Ewf_{I,\varepsilon})$. Then the proof is done.
\end{proof}

As a direct consequence, we obtain:
\begin{corollary}\label{cm}
It is consistent with $\thzfc$ that $\non(\Ewf_{I,\varepsilon})>\bfrak$ for any $I$ and $\varepsilon$.    
\end{corollary}

\section{Open Questions}

We discuss some open questions from this study. With regard to~\autoref{cichonext} and items~\ref{cohen}-\ref{miller}, we do not know the following.

Despite knowing that consistently $\add(\Nwf)<\non(\Ewf_{I,\varepsilon})$ for all $I$ and $\varepsilon$ by~\autoref{cm}, we still ask:

\begin{question}\label{q1}
Is it consistent that $\add(\Ewf_{I,\varepsilon})<\non(\Ewf_{I,\varepsilon})$ for some (or for all) $I$ and $\varepsilon$? Dually, $\non(\Ewf_{I,\varepsilon})<\cof(\Ewf_{I,\varepsilon})$? The same questions apply to $\Swf_{I,\varepsilon}$. 
\end{question}

In relation to~\ref{m5}, we ask:

\begin{question}\label{q2}
Is it consistent that $\cov(\Nwf)<\non(\Ewf_{I,\varepsilon})$  for all $I$ and $\varepsilon$? Dually, the same question about the inequality $\cov(\Ewf_{I,\varepsilon})<\non(\Nwf)$.
\end{question}

One way to attack~\autoref{q1}-\ref{q2} is to see whether $\Por_{I,\varepsilon}$ satisfies the $(\varrho,\rho)$ linked property for some suitable functions $\rho,\sigma\in\baire$ (see~\cite{KO} for more details), but this is not yet clear.

\begin{question}
Is it consistent that $\add(\Nwf)<\add(\Ewf_{I,\varepsilon})$ for some (or for all) $I$ and $\varepsilon$? Dually, $\cof(\Ewf_{I,\varepsilon})<\cof(\Nwf)$? 
\end{question}

\begin{question}
Is it consistent that $\cov(\Nwf)<\cov(\Swf_{I,\varepsilon})$ for some (or for all) $I$ and $\varepsilon$? 
\end{question}

\begin{question}
Is it consistent that $\cov(\Ewf_{I,\varepsilon})$ and $\cov(\Swf_{I,\varepsilon})$ are different for some (or for all) $I$ and $\varepsilon$? The same is asked for uniformities. 
\end{question}

It is a known fact that forcing notions with the Laver property do not increase the covering of $\Ewf$ (see e.g. \cite[Lem.~2.2]{CaraboutE}), so we pose the question:

\begin{question}
Do forcing notions satisfying the Laver property keep $\cov(\Ewf_{I,\varepsilon})$ small? 
\end{question}

Regarding~\autoref{NAbound}, the following is still unknown.

\begin{question}
Is it $\non(\NAwf)< \sup\set{\non(\Ewf_{I,\varepsilon})}{I\in\Ior\mathrm{\ and\ }\varepsilon\in\ell^1_+}$ consistent? Dually, is it $\cov(\NAwf)>\min\set{\cov(\Ewf_{I,\varepsilon})}{I\in\Ior\mathrm{\ and\ } \varepsilon\in\ell^1_+}$ consistent?
\end{question}

In~\cite{mejiamatrix}, Mej\'ia built a forcing model with four cardinal characteristics associated with $\Nwf$ that are pairwise different. The first author~\cite{Car4E} produced a similar model for $\Ewf$, and recently, Brendle, Mej\'ia, and the first author~\cite{BCM2} constructed such a model for $\SNwf$. In this situation, we inquire:

\begin{question}
Is it consistent that for some (or for all) $I$ and $\varepsilon$, the four cardinal characteristics associated with $\Swf_{I,\varepsilon}$ are pairwise different? The same is asked for $\Ewf_{I,\varepsilon}$.  
\end{question}

In \autoref{a3}~\ref{a3c}, we have shown that if $\refin{I}{J}$ and an additional assumption in $I,J,\varepsilon$ is true, then $\Swf_{I,\varepsilon}\subseteq \Swf_{J,\varepsilon}$.

\begin{question}
Is there any reasonable relation $R$ such that if $IRJ$, then $\Swf_{I,\varepsilon}\subseteq \Swf_{J,\varepsilon}$? The same is asked for $\Ewf_{I,\varepsilon}$.
\end{question}

Note that by \autoref{sqrelation_not_working}, the relation $R$ cannot be a directed preorder.

By \autoref{relation:SandE}, $\Swf_{I,\varepsilon}\setminus\Ewf\neq\emptyset$. There is a~natural question about the structure of $\Swf_{I,\varepsilon}$ inside $\Ewf$. We know $\Ewf_{I,\varepsilon}\subseteq \Swf_{I,\varepsilon}\cap\Ewf$, but the reverse inclusion is not clear.  

\begin{question}
Does $\Ewf_{I,\varepsilon}=\Swf_{I,\varepsilon}\cap\Ewf$ hold true for any $\Swf^\star$-contributive $(I,\varepsilon)$?
\end{question}

In \cite{GaMe}, for an ideal $J$ on $\omega$, the families $\Nwf_J^\ast$ and $\Nwf_J$ such that $\Ewf\subseteq\Nwf_J^\ast\subseteq\Nwf_J\subseteq\Nwf$ are introduced and investigated. 

\begin{question}
What is the relation among the $\sigma$-ideals $\Swf_{I,\varepsilon}$, $\Nwf_J^\ast$, and $\Nwf_J$?
\end{question}

\subsection*{Acknowledgments}

The authors express gratitude to Diego Mej\'ia for his helpful input during the completion of this paper and for several stimulating discussions during his time with the Set Theory group in Ko\v sice.

{\small
\bibliography{left}

\begin{thebibliography}{CMRM24}

\bibitem[Bar84]{b}
Tomek Bartoszy\'nski.
\newblock Additivity of measure implies additivity of category.
\newblock {\em Trans. Amer. Math. Soc.}, 281(1):209--213, 1984.

\bibitem[Bar87]{bar}
Tomek Bartoszy\'nski.
\newblock Combinatorial aspects of measure and category.
\newblock {\em Fund. Math.}, 127(3):225--239, 1987.

\bibitem[Bar88]{bartosmall}
Tomek Bartoszy\'nski.
\newblock On covering of real line by null sets.
\newblock {\em Pacific J. Math.}, 131(1):1--12, 1988.

\bibitem[BCM21]{BCM}
J\"{o}rg Brendle, Miguel~A. Cardona, and Diego~A. Mej\'{\i}a.
\newblock Filter-linkedness and its effect on preservation of cardinal
  characteristics.
\newblock {\em Ann. Pure Appl. Logic}, 172(1):Paper No. 102856, 30, 2021.

\bibitem[BCM23]{BCM2}
J\"{o}rg Brendle, Miguel~A. Cardona, and Diego~A. Mej\'ia.
\newblock Separating cardinal characteristics of the strong measure zero ideal.
\newblock Preprint, \href{https://arxiv.org/abs/2309.01931}{ arXiv:2309.01931},
  2023.

\bibitem[BJ95]{bartjud}
Tomek Bartoszy\'{n}ski and Haim Judah.
\newblock {\em Set theory. On the structure of the real line}.
\newblock A K Peters, Ltd., Wellesley, MA, 1995.

\bibitem[BJ95]{bartJudah}
Tomek Bartoszy\'nski and Haim Judah.
\newblock Borel images of sets of reals.
\newblock {\em Real Anal. Exchange}, 20(2):536--558, 1994-95.

\bibitem[Bla10]{blass}
Andreas Blass.
\newblock Combinatorial cardinal characteristics of the continuum.
\newblock In {\em Handbook of set theory. {V}ols. 1, 2, 3}, pages 395--489.
  Springer, Dordrecht, 2010.

\bibitem[BM14]{BrM}
J{\"o}rg Brendle and Diego~Alejandro Mej{\'{\i}}a.
\newblock Rothberger gaps in fragmented ideals.
\newblock {\em Fund. Math.}, 227(1):35--68, 2014.

\bibitem[Bre09]{Brendlecurse}
J\"{o}rg Brendle.
\newblock Forcing and the structure of the real line.
\newblock Lecture notes for the mini-course of the same name at the Universidad
  Nacional of Colombia, 2009.

\bibitem[BS92]{BS1992}
Tomek Bartoszy\'nski and Saharon Shelah.
\newblock Closed measure zero sets.
\newblock {\em Ann. Pure Appl. Logic}, 58(2):93--110, 1992.

\bibitem[BS18]{BART2018}
Tomek Bartoszynski and Saharon Shelah.
\newblock A note on small sets of reals.
\newblock {\em C. R. Math. Acad. Sci. Paris}, 356(11-12):1053--1061, 2018.

\bibitem[Car23]{Car4E}
Miguel~A. Cardona.
\newblock A friendly iteration forcing that the four cardinal characteristics
  of $\mathcal{E}$ can be pairwise different.
\newblock {\em Colloq. Math.}, 173(1):123--157, 2023.

\bibitem[Car24]{CaraboutE}
Miguel~A. Cardona.
\newblock The cardinal characteristics of the ideal generated by the
  {F}$_{\sigma}$ measure zero subsets of the reals.
\newblock {\em Ky\={o}to Daigaku S\=urikaiseki Kenky\=usho K\=oky\=uroku},
  2024.
\newblock To appear, \href{https://arxiv.org/abs/2402.04984}{arXiv:2402.04984}.

\bibitem[CM19]{CM}
Miguel~A. Cardona and Diego~A. Mej\'{\i}a.
\newblock On cardinal characteristics of {Y}orioka ideals.
\newblock {\em MLQ}, 65(2):170--199, 2019.

\bibitem[CM23]{CM23}
Miguel~A. Cardona and Diego~A. Mej\'ia.
\newblock Localization and anti-localization cardinals.
\newblock {\em arXiv:2305.03248}, 2023.

\bibitem[CMRM24]{CMR2}
Miguel~A. Cardona, Diego~A. Mej\'{\i}a, and Ismael~E. Rivera-Madrid.
\newblock Uniformity numbers of the null-additive and meager-additive ideals.
\newblock Preprint, \href{https://arxiv.org/abs/2401.15364}{arXiv:2401.15364},
  2024.

\bibitem[GM23]{GaMe}
Viera Gavalová and Diego~Alejandro Mejía.
\newblock Lebesgue measure zero modulo ideals on the natural numbers.
\newblock {\em J. Symb. Log.}, pages 1--30, 2023.
\newblock Published online,
  \href{https://www.cambridge.org/core/journals/journal-of-symbolic-logic/article/abs/lebesgue-measure-zero-modulo-ideals-on-the-natural-numbers/CF264BE3BADDB75C562F1AC38FEFF8A3}{doi:10.1017/jsl.2023.97}.

\bibitem[GMS16]{GMS}
Martin Goldstern, Diego~Alejandro Mej{\'{\i}}a, and Saharon Shelah.
\newblock The left side of {C}icho\'n's diagram.
\newblock {\em Proc. Am. Math. Soc.}, 144(9):4025--4042, 2016.

\bibitem[Kad00]{kada}
Masaru Kada.
\newblock More on {C}icho\'{n}'s diagram and infinite games.
\newblock {\em J. Symbolic Logic}, 65(4):1713--1724, 2000.

\bibitem[KO14]{KO}
Shizuo Kamo and Noboru Osuga.
\newblock Many different covering numbers of {Y}orioka's ideals.
\newblock {\em Arch. Math. Logic}, 53(1-2):43--56, 2014.

\bibitem[Mej13]{mejiamatrix}
Diego~Alejandro Mej{\'{\i}}a.
\newblock Matrix iterations and {C}icho\'n's diagram.
\newblock {\em Arch. Math. Logic}, 52(3-4):261--278, 2013.

\bibitem[Mej19]{mejiavert}
Diego~A. Mej\'{\i}a.
\newblock Matrix iterations with vertical support restrictions.
\newblock In {\em Proceedings of the 14th and 15th {A}sian {L}ogic
  {C}onferences}, pages 213--248. World Sci. Publ., Hackensack, NJ, 2019.

\bibitem[Mil82]{Mil82}
Arnold~W. Miller.
\newblock A characterization of the least cardinal for which the {B}aire
  category theorem fails.
\newblock {\em Proc. Amer. Math. Soc.}, 86(3):498--502, 1982.

\bibitem[She95]{shelahnulladd}
Saharon Shelah.
\newblock Every null-additive set is meager-additive.
\newblock {\em Israel J. Math.}, 89(1-3):357--376, 1995.

\bibitem[She00]{Sh00}
Saharon Shelah.
\newblock Covering of the null ideal may have countable cofinality.
\newblock {\em Fund. Math.}, 166(1-2):109--136, 2000.

\bibitem[Voj93]{vo93}
Peter Vojt\'{a}\v{s}.
\newblock Generalized {G}alois-{T}ukey-connections between explicit relations
  on classical objects of real analysis.
\newblock In {\em Set theory of the reals ({R}amat {G}an, 1991)}, volume~6 of
  {\em Israel Math. Conf. Proc.}, pages 619--643. Bar-Ilan Univ., Ramat Gan,
  1993.

\end{thebibliography}
\bibliographystyle{alpha}
}


\end{document}